%% file: main.tex
\RequirePackage{fix-cm}
\documentclass[smallextended]{svjour3}
\smartqed
\usepackage[algo2e,linesnumbered,vlined,ruled]{algorithm2e}
\usepackage{graphicx,graphics,subfigure}
\usepackage{hyperref,url}
\usepackage{cancel}
\usepackage{enumerate}
\usepackage{bbm}
\usepackage{comment}
\usepackage{times,hyperref}
\usepackage{amsmath,amsxtra,amsfonts,amscd,amssymb,bm}
\usepackage[normalem]{ulem}
\usepackage{exscale}
\usepackage{tikz}
\usepackage{float}
\usepackage{algorithm}
\usepackage{algorithmic}
\usepackage{multirow}
\usepackage{multicol}
\usepackage{makecell}
\usepackage{epstopdf}
\usepackage{tabularx}

\usepackage{tikz}
\usetikzlibrary{positioning}
\usepackage{xcolor}
\usepackage{color}
\usepackage{colortbl}
\usepackage{multirow}

\newcommand{\be}{\begin{equation}}
\newcommand{\ee}{\end{equation}}
\newcommand{\bee}{\begin{equation*}}
\newcommand{\eee}{\end{equation*}}
\newcommand{\bea}{\begin{eqnarray}}
\newcommand{\eea}{\end{eqnarray}}
\newcommand{\beaa}{\begin{eqnarray*}}
\newcommand{\eeaa}{\end{eqnarray*}}

\input{command.tex}

\usepackage{float}

\allowdisplaybreaks

\begin{document}

\title{On the Optimal Lower and Upper Complexity Bounds for a Class of Composite Optimization Problems}

\titlerunning{Optimal Complexity Bounds for a Class of Composite Optimization Problems}

\author{Zhenyuan Zhu \and Fan Chen \and Junyu Zhang \and Zaiwen Wen  
}

\institute{Zhenyuan Zhu \at
              Beijing International Center for Mathematical Research, Peking University, Beijing, China\\
              \email{zhenyuanzhu@pku.edu.cn} 
           \and
           Fan Chen \at 
              School of Mathematical Science, Peking University, Beijing, China\\
              \email{chern@pku.edu.cn}
          \and
          Corresponding Author: Junyu Zhang \at
          Department of Industrial Systems Engineering and Management, National University of Singapore, Singapore, Singapore\\
          \email{junyuz@nus.edu.sg}
          \and
          Zaiwen Wen \at 
              International Center for Mathematical Research, Center for Machine Learning Research and College of Engineering, Peking University, Beijing, China\\
              \email{wenzw@pku.edu.cn}
}

\date{}

\maketitle

\begin{abstract}
We study the optimal lower and upper complexity bounds for finding approximate solutions to the composite problem $\min_x\ f(x)+h(Ax-b)$, where $f$ is smooth and $h$ is convex. Given access to the proximal operator of $h$, for strongly convex, convex, and nonconvex $f$, 
we design efficient first order algorithms with complexities  $\tilde{O}\left(\kappa_A\sqrt{\kappa_f}\log\left(1/{\epsilon}\right)\right)$, $\tilde{O}\left(\kappa_A\sqrt{L_f}D/\sqrt{\epsilon}\right)$, and $\tilde{O}\left(\kappa_A L_f\Delta/\epsilon^2\right)$, respectively. Here, $\kappa_A$ is the condition number of the matrix $A$ in the composition, $L_f$ is the smoothness constant of $f$, and $\kappa_f$ is the condition number of $f$ in the strongly convex case. $D$ is the initial point distance and $\Delta$ is the initial function value gap. Tight lower complexity bounds for the three cases are also derived and they match the upper bounds up to logarithmic factors, thereby demonstrating the optimality of both the upper and lower bounds proposed in this paper.
\end{abstract}

\section{Introduction}
In this paper, we consider the composite optimization problem:
\begin{equation}\label{comp-prob}
\min_x\ F(x)\defeq f(x)+h(Ax-b),
\end{equation}
where $f:\mathbb{R}^n\mapsto\mathbb{R}$ is a differentiable function whose gradient is $L_f$-Lipschitz continuous, $h:\mathbb{R}^n\mapsto\mathbb{R}$ is a proper closed convex function whose proximal operator can be efficiently evaluated, $A\in\RR^{m\times n}$ is a matrix with full row rank and its minimal singular value is $\sigmamin$. We denote the condition number of matrix $A$ as $\kappa_A\defeq\frac{\|A\|}{\sigmamin}$. 
Problem \eqref{comp-prob} appears in many practical applications, such as distributed optimization \cite{shi2015extra,chang2016asynchronous,hong2017prox}, 
signal and image processing \cite{yang2011alternating,zhu2008efficient,chambolle2011first}, network optimization \cite{feijer2010stability} and optimal control \cite{bertsekas2012dynamic}. In this paper, we aim to figure out the tight upper and lower complexity bounds for all three cases of strongly convex, convex, and non-convex function $f$. 

When $h(\cdot)$ is chosen to be the indicator function $h(\cdot)= \mathbb{I}_{\{0\}}(\cdot)$, the problem becomes the linear equality constrained problem
\begin{equation}\label{linear-equality}
\min_x f(x),\quad \st Ax=b.
\end{equation}
More generally, when $h(\cdot)$ is chosen to be $h(\cdot)= \mathbb{I}_{\mathcal{K}}(\cdot)$ where $\mathcal{K}$ is a proper cone, the problem becomes the conic inequality constrained problem
\begin{equation}\label{conic-constrained}
\min_x\ f(x), \quad \st Ax-b\in\mathcal{K}.
\end{equation}
Alternatively, when $h(\cdot)=\nrm{\cdot}_{*}$ is a norm function, the problem becomes the norm regularized problem
\begin{equation}\label{regularized-prob}
\min_x\ f(x)+ \|Ax-b\|_{*}.
\end{equation}

The motivation of this paper originates from the study of the complexity of the non-convex optimization with linear equality constraints \eqref{linear-equality}, which corresponds to the special case of \eqref{comp-prob} when $h(\cdot)=\mathbb{I}_{\{0\}}(\cdot)$ and $f(x)$ is non-convex. This problem has received considerable attention in the literature. The authors of \cite{kong2019complexity} combine the proximal point method and the quadratic penalty method, archiving an $O(\epsilon^{-3})$ complexity for finding an $\epsilon$-stationary point. Under the assumption that the domain is bounded, the authors of \cite{kong2023iteration} develop an algorithm that solves proximal augmented Lagrangian subproblem by accelerated composite gradient method, resulting in a complexity of $\tilde O(\epsilon^{-2.5})$. 

Several recent works have further improved the complexity to $O(\epsilon^{-2})$, including \cite{hong2017prox,sun2019distributed,zhang2020proximal,zhang2022global}. The key factor contributing to this improvement is the incorporation of the minimal nonzero singular value in the analysis of complexity bounds. Specifically, the condition number of $A$ is required to be bounded by some constant. In \cite{hong2017prox}, a proximal primal-dual algorithm is developed, while the dependence on $\kappa_A$ is very large. In \cite{sun2019distributed}, a novel optimal method for distributed optimization is proposed. Yet its dependence on Chebyshev acceleration prevents its generalization to problems with $h(\cdot)\neq \mathbb{I}_{\{0\}}(\cdot)$. The authors of \cite{zhang2020proximal} focus on the problem with an extra box constraint while it requires assuming the complementarity property, which is not easy to check in advance. The subsequent work \cite{zhang2022global} relaxes the assumption, but the dependence on other parameters remains unclear.

Therefore, it is natural to ask what are the complexity upper and lower bounds of first-order methods for the non-convex linearly constrained problem \eqref{linear-equality}. To answer this question, we consider the general form problem \eqref{comp-prob} and utilize the norm of the proximal gradient mapping $L_f\left\| x-\proxhAx\left(x-\frac{1}{2L_f}\nabla f(x)\right)\right\|$ as the convergence measure. When $h$ reduces to an indicator function $h(\cdot)=\mathbb{I}_{\{0\}}(\cdot)$ which corresponds to the linear equality constrained problem \eqref{linear-equality}, this error measure reduces to the norm of the projected gradient mapping $L_f\left\| x-\mathcal{P}_{\{Ax=b\}}\left(x-\frac{1}{2L_f}\nabla f(x)\right)\right\|$. We rigorously define the first-order linear span algorithm class and construct a hard instance to establish a lower bound of $\Omega\left(\kappa_A L_f\Delta/\epsilon^2\right)$ where $\Delta$ quantifies the gap in the objective function at the initial point. To demonstrate the tightness of the lower bound, we design an efficient algorithm based on the idea of accelerated proximal point algorithm (APPA), which achieves the bound of $\tilde{O}\left(\kappa_A L_f\Delta/\epsilon^2\right)$.

The APPA approach solves a non-convex problem by successively solving a sequence of strongly convex subproblems. Hence, the optimality of the complexity for solving problem \eqref{comp-prob} with strongly convex $f$ is also crucial. One straightforward approach is to introduce an additional variable $y=Ax-b$ and apply ADMM to solve the linear constrained problem where one function in the objective is strongly convex and the other is convex. For example, a sublinear convergence rate is obtained in \cite{cai2017convergence} and \cite{lin2015sublinear}. The lower bound is examined by constructing a challenging linear equality constrained problem in \cite{ouyang2021lower}. Assuming a strongly convex parameter $\mu_f$, the derived lower bound is $\Omega\left(\frac{\|A\|\|\lambda^\star\|}{\mu_f\sqrt{\epsilon}}\right)$ where $\lambda^\star$ is the optimal dual variable. Remarkably, this dominant term matches the upper bound provided in \cite{xu2021iteration} up to logarithmic factors. 

If we further incorporate the minimal nonzero singular value of $A$ into the complexity bound, it is possible to ``break'' the lower bound established in \cite{ouyang2021lower}. For example, the ADMM and a set of primal-dual methods are proved to have linear convergence rates \cite[etc.]{lin2015global,zhu2022unified}. Furthermore, for the linear equality constrained problem \eqref{linear-equality} with strongly convex $f$, the algorithms proposed in \cite{zhu2022unified} and \cite{salim2022optimal} achieve complexity of $O\left((\kappa_A^2+\kappa_f)\log(1/\epsilon)\right)$ and $O\left(\kappa_A\sqrt{\kappa_f}\log(1/\epsilon)\right)$, respectively. However, \cite{salim2022optimal} utilizes Chebyshev iteration as a sub-routine, and is hence restricted to the linear equality constrained problems. An $\Omega\left(\kappa_A\sqrt{\kappa_f}\log(1/\epsilon)\right)$ lower bound is also claimed for the linear equality constrained case in \cite{salim2022optimal}. 

For completeness, we also discuss the case with general convex $f$, which has been extensively investigated. The authors of \cite{ouyang2015accelerated} and \cite{xu2017accelerated} design algorithms that converge to an $\epsilon$-optimal solution with a complexity of $O(\epsilon^{-1})$. \cite{zhu2022unified} proposes a unified framework that covers several well-known primal-dual algorithms and establishes an ergodic convergence rate of $O(\epsilon^{-1})$. The lower bound is provided in \cite{ouyang2021lower} and the dominant term is $\Omega\left({\|A\|\|x^\star\|\|\lambda^\star\|}/{\epsilon}\right)$ which matches the upper bounds provided in \cite{ouyang2015accelerated}. Recent researches also establish complexity results better than $O(\epsilon^{-1})$ under stronger assumptions. On the linear constrained problem with $O(1)$ constraints, \cite{xu2022first} utilizes the cutting-plane method as the subroutine of ALM and proves a dimension-dependent complexity of $O(1/\sqrt{\epsilon})$. For a class of bilinear and smooth minimax problems (which include linear constrained problems), \cite{song2020breaking} proposes an algorithm that achieves $O\left(1/\sqrt{\epsilon}\right)$ complexity, but its dependence on other constants remains suboptimal.
Under our setting, we demonstrate that both lower and upper bounds can achieve $\tilde{\Theta}(\kappa_AD\sqrt{L/\epsilon})$ by deducing from the strongly convex case.\vspace{0.2cm}

\textbf{Contribution}. Our results are listed in Table \ref{table:ours} and our main contributions are summarized as follows. 
\begin{itemize}
\item Under the assumption of bounded $\kappa_A$, we establish the complexity lower bounds for solving \eqref{comp-prob} with non-convex, strongly convex, and convex $f$, within the first-order linear span algorithm class, by constructing three hard linear equality constrained instances.
\item By exploiting the idea of the (accelerated) proximal point algorithm, we design efficient algorithms to solve problem \eqref{comp-prob} in all three cases.
\item Under the assumption of bounded $\kappa_A$, we prove that the complexities of our algorithms match the lower bounds up to logarithmic factors. Therefore, these complexities are optimal.
\item For the special case of linear equality constrained problem \eqref{linear-equality}, we prove that the full row rank assumption can be removed. The upper bounds keep unchanged except that $\kappa_A$ is replaced by $\underline{\kappa}_A=\|A\|/\sigmin$, with $\sigmin$ being the minimum nonzero singular value.
\end{itemize}

\begin{table}[ht]
\centering
\begin{tabular}{c|c|c}
\hline\hline
Setting & Optimality Measure & Complexity\\\hline\hline
\multirow{2}{*}{Strongly Convex} & \multirow{2}{*}{$\|x-x^\star\|$} & $\tilde{O}\left(\kappa_A\sqrt{\kappa_f}\log\left(1/{\epsilon}\right)\right)$\\\cline{3-3}
&&\cellcolor{lightgray!60}{$\Omega\left(\kappa_A\sqrt{\kappa_f}\log\left(1/{\epsilon}\right)\right)$}\\\hline
\multirow{2}{*}{Convex} & \multirow{2}{*}{$f(x)+h_\rho(Ax-b)-\min_x F(x)$} & $\tilde{O}\left(\kappa_A\sqrt{L_f}D/\sqrt{\epsilon}\right)$\\\cline{3-3}
&&\cellcolor{lightgray!60}{$\Omega\left(\kappa_A\sqrt{L_f}D/\sqrt{\epsilon}\right)$}\\\hline
\multirow{2}{*}{Non-convex} & \multirow{2}{*}{$L_f\left\| x-\proxhAx\left(x-\frac{1}{2L_f}\nabla f(x)\right)\right\|$} & $\tilde{O}\left(\kappa_A L_f\Delta/\epsilon^2\right)$\\\cline{3-3}
&&\cellcolor{lightgray!60}{$\Omega\left(\kappa_A L_f\Delta/\epsilon^2\right)$}\\\hline\hline
\end{tabular}
\caption{Summary of our results. The gray cells represent the lower bounds. $x^\star$: the optimal solution, $x_0\in\dom F$: the initial point, $h_\rho(\cdot)$: a surrogate function of $h(\cdot)$ (defined in \eqref{def:hrho}), $\prox$: the proximal operator, $L_f$: Lipschitz constant of $\nabla f(x)$, $\mu_f$: strong convexity parameter of $f(x)$, $\kappa_f=L_f/\mu_f$, $\kappa_A=\|A\|/\sigmamin$, $F(x_0)-F(x^*)\leq \Delta$, $\|x_0-x^*\|\leq D$.}
\label{table:ours}
\end{table}

\textbf{Related works}. For this paper, there are several closely related works.  The first one that we would like to discuss is \cite{song2020breaking}. It is designed for a class of smooth bilinear minimax problems, which includes convex linearly constrained problems after proper reformulation. With some extra computation, the results of \cite{song2020breaking} indicate an $O\big(\kappa_A(D+D^*)\sqrt{L_f/\epsilon}\big)$ complexity, where $D^*$ upper bounds the norm of the optimal dual variables. The additional $O\big(\kappa_AD^*\sqrt{L_f/\epsilon}\big)$ term makes their algorithm suboptimal for our problem class. Second, for linear equality constrained problems, their results require the constraint matrix $A$ to be full rank, while our result does not require this property. In our paper, we provide a tight lower complexity bound for this problem class and an optimal first-order algorithm that achieves the lower bound.

The second closely related work is \cite{salim2022optimal}, which derived an $O\left(\kappa_A\sqrt{\kappa_f}\log(1/\epsilon)\right)$ complexity for the smooth strongly convex linear equality constrained problem \eqref{linear-equality}. Let $P(\cdot)$ be some properly selected Chebyshev polynomial, they proposed an accelerated algorithm with Chebyshev iteration based on the equivalence between $Ax=b$ and $\sqrt{P(A^\top A)}x=\sqrt{P(A^\top A)}x^*$. However, this technique is restricted to linear equality constrained problems as the aforementioned equivalence failed for inequality contraints $\big(h(\cdot)=\mathbb{I}_{\mathbb{R}^n_+}(\cdot)\big)$ and more general $h$. Moreover, their algorithm requires the smoothness of the objective function over the whole space, and they do not allow any constraints other than the linear equality ones. Hence their method cannot handle general non-smooth $h(\cdot)\neq\mathbb{I}_{\{0\}}(\cdot)$ by introducing $y=Ax-b$ and sloving an equality constrained problem. This is because the new $h(y)$ term violates both the global smoothness and joint strong convexity assumptions in their paper. It is worth mentioning that the $\Omega\left(\kappa_A\sqrt{\kappa_f}\log(1/\epsilon)\right)$ lower bound in \cite{salim2022optimal} is indeed a valid lower bound for our problem. However, since the construction of our lower bound for the general convex case requires the lower bound for the strongly convex case as an intermediate step, we provide a self-contained lower bound for the strongly convex case in this paper. The hard instance and the structure of the zero chain in our example are both different from those in \cite{salim2022optimal}.

The last related work that we would like to discuss is \cite{sun2019distributed} where an optimal first-order algorithm for non-convex distributed optimization is proposed. Like \cite{salim2022optimal}, their results require the global smoothness of the objective function and the Chebyshev acceleration and is hence hard to generalize to $h(\cdot)\neq\mathbb{I}_{\{0\}}(\cdot)$. A lower bound is also proposed in \cite{sun2019distributed} paper for non-convex distributed optimization. However, it is not clear how their lower bound can be reduced to our case and both their lower and upper bounds include an unusual dependence on the squared norm of the initial gradient. Therefore, instead of trying to reduce from their results, we construct a new hard instance that lower bounds the complexity for making the proximal gradient mapping small. 

\vspace{0.2cm}
\textbf{Organization.}~ 
Section \ref{sec:pre} provides fundamental information and defines the quantities to measure the suboptimality of an approximate solution. In Section \ref{sec:upper-bound}, we discuss the upper bounds for each of the three cases. For the sake of readability, we present an efficient algorithm for the strongly convex case first. Then, we utilize this algorithm to solve the proximal point subproblems of the non-convex case. We also utilize this algorithm to solve the general convex case by adding $O(\epsilon)$-strongly convex perturbation and obtain a corresponding upper bound. Thus, we follow the order from strongly convex to non-convex to convex in the discussion. Then in Section \ref{sec:lower-bound}, we present formal definitions of problem classes and algorithm classes and provide the corresponding lower bounds. Finally, in Section \ref{sec:linear-equality}, we conduct a detailed analysis of the linear equality constrained problem. In particular, we relax the full rank assumption on matrix $A$. A direct acceleration for the general convex case is also presented in the last section, instead of adding strongly convex perturbations.\vspace{0.2cm}

\textbf{Notations.} We denote the Euclidean inner product by $\langle \cdot,\cdot \rangle$ and the Euclidean norm by $\|\cdot\|$. 
Let $f(\cdot,\cdot):\RR^m\times \RR^n\mapsto \RR$ be a differentiable function of two variables. To denote the partial gradient of $f$ with respect to the first (or second) variable at the point $(x,y)$, we use $\nabla_x f(x,y)$ (or $\nabla_y f(x,y)$). The full gradient at $(x,y)$ is denoted as $\nabla f(x, y)$, where $\nabla f(x,y) = (\nabla_x f(x,y),\nabla_y f(x,y))$.
Suppose $\mathcal{M}$ is a closed convex set, we use $\mathcal{P}_{\mathcal{M}}(\cdot)$ to represent the projection operator onto $\mathcal{M}$. For any matrix $A\in\RR^{m\times n}$, we use $\mathcal{R}(A)$ to denote its range space. For any vector $x\in\RR^n$, we use $\supp\left\{x\right\}$ to denote the support of $x$, i.e., $\supp\left\{x\right\} \defeq \left\{i\in[n]\mid x_i \neq 0\right\}$, where $[n]=\{1,\cdots,n\}$. 
The indicator function for a set $\mathcal{S}$ is defined as $\mathbb{I}_{\mathcal{S}}(x)=0$ if $x\in\mathcal{S}$ and $\mathbb{I}_{\mathcal{S}}(x)=+\infty$ if $x\notin\mathcal{S}$. The identity matrix in $\RR^{d\times d}$ is denoted as $I_d$. The positive part of a real number is represented as $[\cdot]_+$, defined as $[x]_+=\max\{x,0\}$.

\section{Preliminaries}\label{sec:pre}

\subsection{Basic facts and definitions}

\noindent\textbf{Lipschitz Smoothness.} For a differentiable function $f(x)$, we say it is $L_f$-smooth if 
    \[
    \|\nabla f(x)-f(x^\prime)\|\leq L_f\|x-x^\prime\|, \quad \forall x,x^\prime.
    \] 

\noindent\textbf{Strong Convexity.}
    For a positive constant $\mu_f> 0$, we say $f(x)$ is $\mu_f$-strongly convex if $f(x)-\frac{\mu_f}{2}\|x\|^2$ is convex, and it is $\mu_f$-strongly concave if $-f(x)$ is strongly convex. \vspace{0.05cm}

\noindent\textbf{Conjugate Function.}
    The conjugate function of a function $f(\cdot)$ is defined as
    \[
    f^\star(y)=\sup_x\{x^\T y-f(x)\}.
    \] 
It is well-known that $f^\star$ is convex. Furthermore, when $f$ is assumed to be strongly convex or smooth, its conjugate function $f^\star$ has the following properties \cite{hiriart1996convex}.
\begin{lemma}\label{lemma:conjugate}
If $f$ is $\mu_f$-strongly convex, then its conjugate $f^\star$ is $\frac1{\mu_f}$-smooth. If $f$ is $L_f$-smooth, then its conjugate $f^\star$ is $\frac{1}{L_f}$-strongly convex.
\end{lemma}

\noindent\textbf{Proximal Operator.}
For a proper closed convex function $h$, the corresponding proximal operator is defined as
\[
\prox_{h}(x)\defeq\argmin_u \left\{h(u)+\frac12\|x-u\|^2\right\}.
\]

\subsection{Suboptimality measure}
To study the upper and lower complexity bounds of first-order methods for problem \eqref{comp-prob}, it is essential to define the appropriate suboptimality measures for the strongly convex, convex, and non-convex cases, respectively. In the following sections, we will present these measures one by one.\vspace{0.2cm}

\noindent\textbf{Strongly convex case.}\,\, In this case, the optimal solution is unique. We can define the suboptimal measure for a point $x$ as its squared distance to the optimal solution
\begin{equation}\label{def:SC-subopt}
\suboptsc(x) \defeq \|x-x^\star\|^2.
\end{equation}

\noindent\textbf{Non-convex case.}\,\, In the non-convex case, we define the suboptimal measure as
\[
    \suboptnc(x)\defeq L_f\left\| x-\proxhAx\left(x-\frac{1}{2L_f}\nabla f(x)\right)\right\|,
\]
which measures the violation of the first-order optimality condition. 
It is worth pointing out that the operator $\proxhAx$ is not needed in the algorithm. We only utilize it as a formal measure but do not need to directly evaluate it in our algorithms.\vspace{0.2cm}

\noindent\textbf{Convex case.}\,\, The convex case is a little bit complicated. The standard suboptimality measure $F(x)-\inf_{x'} F(x')$  may not fit our setting, since $h(Ax-b)=+\infty$ could happen when $Ax-b\not\in \dom h$. For example, for linearly constrained problems where $h(Ax-b)=\mathbb{I}_{\{Ax=b\}}(x)$, typically first-order methods only guarantee returning  solutions with small constraint violation instead of exact constraint satisfaction. Thus the objective function gap will always be $+\infty$. To handle this issue, we propose the following suboptimality measure
\[
\suboptc(x)\defeq f(x)+h_\rho(Ax-b)-\min_{x'}\left\{f(x')+h(Ax'-b)\right\},
\]
where $h(\cdot)$ is replaced by the surrogate function $h_\rho(\cdot)$ given by
\begin{equation}\label{def:hrho}
h_\rho(z)\defeq \sup_{\|y\|_2\leq \rho}\{z^\T y-h^\star(y)\}.
\end{equation}
For any $\rho\in(0,+\infty)$. $h_\rho(x)$ can be viewed as a Lipschitz approximation of $h(x)$, as the following lemma indicates.

\begin{lemma}
\label{lemma:hrho}
For any $\rho\in(0, +\infty)$, the followings hold
\begin{enumerate}
\item $h_\rho(x)\leq h(x)$, and $h_{\rho}(x) \to h(x)$ as $\rho\to+\infty$. \label{hrho1}
\item $h_\rho(x)$ is $\rho$-Lipschitz continuous, and if $h(x)$ is $\rho$-Lipschitz continuous, then $h_\rho(x)\equiv h(x)$.\label{hrho-Lip} \label{hrho2}
\item If $h(x)=\mathbb{I}_{\mathcal{K}}(x)$ where $\mathcal{K}$ is a proper cone, we have $h_\rho(x)= \rho\|\mathcal{P}_{\mathcal{K}^\circ}(x)\|$. \label{hrho-cone}
\end{enumerate}
\end{lemma}

\begin{proof}
Part \ref{hrho1}. 
$h_\rho(x)=\sup_{\|x\|\leq\rho}\{x^\T y-h^\star(x)\}\leq \sup_x \{x^\T y-h^\star(x)\}\leq h(x)$.

Part \ref{hrho2}. For any $x,x^\prime$, we have
\begin{align*}
    h_\rho(x)&=\sup_{\|y\|\leq\rho} \{y^\T(x-x^\prime)+y^\T x^\prime-h^\star(y)\}\\
    &\leq \sup_{\|y\|\leq\rho} y^\T(x-x^\prime) + \sup_{\|y\|\leq\rho}\{y^\T x^\prime-h^\star(y)\}\\
    &\leq \rho\|x-x^\prime\|+h_\rho(x^\prime),
\end{align*}
which implies that $h_\rho(x)$ is $\rho$-Lipschitz continuous.

For any $x$, let $y^\star(x)=\argmax_y \{x^\T y-h^\star(y)\}$, then we have $y^\star(x)\in\partial h(x)$. It holds $\|y^\star(x)\|\leq \rho$ if $h(x)$ is $\rho$-Lipschitz continuous. Hence, we have $h_\rho(x)\equiv h(x)$.

Part \ref{hrho-cone}.
When $h(x)=\mathbb{I}_{\mathcal{K}}(x)$, the conjugate function is $h^\star(y)=\mathbb{I}_{\mathcal{K}^\circ}(y)$ where $\mathcal{K}^\circ$ is the polar cone of $\mathcal{K}$. Therefore, we have
\[h_\rho(x)=\sup\limits_{\|y\|\leq\rho, y\in\mathcal{K}^\circ} x^\T y=x^\T \left(\frac{\rho}{\|\mathcal{P}_{\mathcal{K}^\circ}(x)\|} \mathcal{P}_{\mathcal{K}^\circ}(x)\right) =\rho\|\mathcal{P}_{\mathcal{K}^\circ}(x)\|.\]\qed
\end{proof}

\begin{remark}
For norm regularized problems \eqref{regularized-prob} where $h(x)=\|x\|_*$ is an arbitrary norm, let $\nrm{\cdot}^*$ be its dual norm, then the conjugate function $h^\star(y)=\mathbb{I}_{\nrm{y}^*\leq1}(y)$. As long as $\rho\nrm{x}^*\geq \nrm{x}, \forall x$, it holds that $$h_\rho(x)=\sup\limits_{\|y\|\leq\rho, \nrm{y}^*\leq1} x^\T y =\sup\limits_{\nrm{y}^*\leq1} x^\T y=\nrm{x}_*=h(x).$$
In particular, when $h(\cdot)=\|\cdot\|$, it is sufficient to take $\rho\geq 1$.
\end{remark}

\begin{remark}
For the linear inequality constrained problems \eqref{linear-equality} with $h(x)=\mathbb{I}_{\{x\leq0\}}$, Lemma \ref{lemma:hrho} indicates that $h_{\rho}(Ax-b)=\rho\nrm{[Ax-b]_+}.$ Therefore, we have
\begin{align*}
    \suboptc(x)=f(x)-f(x^\star)+\rho\nrm{[Ax-b]_+},
\end{align*}
Let $x^\star$ and $\lambda^\star$ be the optimal primal and dual solutions respectively, it is a standard lemma that (see e.g. \cite[Lemma 3]{zhu2022unified}) when $\rho\geq 2\nrm{\lambda^\star}$, we have 
\begin{align*}
    \max\set{ \abs{f(x)-f(x^\star)}, \rho\nrm{[Ax-b]_+} }\leq 2\suboptc(x).
\end{align*}
Similarly, for linear equality constrained problems with $h(x)=\mathbb{I}_{\{0\}}(x)$, we have 
$h_\rho(Ax-b)=\rho\|Ax-b\|$. Then $\max\set{ \abs{f(x)-f(x^\star)}, \rho\nrm{[Ax-b]} }\leq 2\suboptc(x)$, if $\rho\geq\|\lambda^\star\|$. Therefore, $\suboptc(x)$ essentially agrees with the widely used suboptimality measure of convex linear constrained problems, see e.g. \cite{xu2021first,zhu2022unified,hamedani2021primal}. Furthermore, we can also cover the problem formulation that with both equality and inequality constraints studied in \cite{zhang2022global,zhang2020proximal}, if we let $h(\cdot)=\mathbb{I}_{\{0\}\times\{x\leq0\}}(\cdot)$.
\end{remark}

\subsection{Nesterov’s accelerated gradient descent}
In this paper, we will frequently use Nesterov’s accelerated gradient descent (AGD) as a subroutine, as stated in Algorithm \ref{algo:AGD}. It is the optimal first-order algorithm for the smooth strongly convex optimization problems \cite{nesterov2018lectures}.  

\newcommand{\AGD}{\mathsf{AGD}}
\begin{algorithm2e}[H]\label{algo:AGD}
	\caption{$\AGD(\Psi,y_0,\delta)$}
	\textbf{Input:} 
        objective function $\Psi$, Lipschitz constant $L$, strong convexity parameter $\mu$, initial point $y_0$ and tolerance $\delta>0$.\vspace{0.05cm}\\
    \textbf{Initialize:} $\kappa\leftarrow \frac L\mu, \theta\leftarrow\frac{\sqrt \kappa-1}{\sqrt \kappa+1}, y_{-1}\leftarrow y_0, k\leftarrow 0$.\\
    \Repeat{$\nrm{\nabla\Psi(y_T)}\leq \mu\delta$}{
    $\tilde y_{k+1}\leftarrow y_k+\theta(y_k-y_{k-1})$.\\
    $y_{k+1}\leftarrow \tilde y_{k+1}-\frac1L\nabla \Psi(\tilde y_{k+1})$.\\
    $k\leftarrow k+1$.
    }
    \textbf{Output:} $y_T$.
\end{algorithm2e}

The following proposition is a simple corollary of \cite[Theorem 2.2.2]{nesterov2018lectures}.
\begin{proposition}\label{prop:AGD}
Assume that $\Psi$ is $L$-smooth and $\mu$-strongly convex. Denote $\kappa=\frac L\mu$ and $y^\star=\argmin_y\Psi(y)$. Algorithm \ref{algo:AGD} $\AGD(\Psi,y_0,\delta)$ terminates in
\begin{align*}
    T\leq 2\sqrt{\kappa}\log\paren{\frac{2\kappa\nrm{y_0-y^\star}}{\delta}}
\end{align*}
iterations, and the output $y_T$ satisfies $\nrm{y_T-y^\star}\leq \delta$.
\end{proposition}

\section{Upper bounds}\label{sec:upper-bound}

\subsection{Strongly convex case: Intuition}\label{sec:SC-ub}
Though $\prox_h(x)$ is easy to evaluate, $h(Ax-b)$ may not have efficient proximal mapping. Therefore, by utilizing the conjugate function of $h$, we decouple the composite structure by reformulating \eqref{comp-prob} as a convex-concave saddle point problem 
\begin{equation}\label{eqn:SC-saddle}
\min_x\max_\lambda \ \mathcal{L}(x,\lambda)\defeq f(x)+\lambda^\T (Ax-b)-h^\star(\lambda).
\end{equation}
Switching the order of $\min$ and $\max$, and minimizing with respect to $x$, we obtain the dual problem of \eqref{comp-prob}:
\begin{equation}\label{eqn:comp-dual}
\max_\lambda\ \Phi(\lambda) \defeq -f^\star(-A^\T \lambda)-b^\T\lambda-h^\star(\lambda).
\end{equation}
The following lemma illustrates that the dual problem is strongly concave.

\begin{lemma}\label{lemma:Phi-SC}
    $\Phi(\lambda)$ is $\muphi$-strongly concave with $\muphi\defeq \sigmamin^2/L_f$.
\end{lemma}

\begin{proof}
Note that $f(x)$ is $L_f$-smooth and $\mu_f$-strongly convex, Lemma \ref{lemma:conjugate} indicates that $f^\star(\cdot)$ is $\frac 1{\mu_f}$-smooth and $\frac 1{L_f}$-strongly convex. Denote $\tilde f(\lambda)\defeq f^\star(-A^\T\lambda)$, then for any $\lambda,\lambda^\prime\in\mathbb{R}^m$, we have 
\begin{align*} 
    \tilde f(\lambda^\prime)-\tilde f(\lambda)
    & \geq \langle \nabla f^\star(-A^\T \lambda), -A^\T \lambda^\prime+A^\T \lambda\rangle +\frac 1{2L_f} \|A^\T \lambda-A^\T \lambda^\prime\|^2\\
    & = \langle \nabla \tilde f(\lambda), \lambda^\prime-\lambda\rangle+\frac 1{2L_f} \|A^\T \lambda-A^\T \lambda^\prime\|^2\\
    & \geq \langle \nabla \tilde f(\lambda), \lambda^\prime-\lambda\rangle+\frac{\sigmamin^2}{2L_f}\|\lambda-\lambda^\prime\|^2,
\end{align*}
which implies that $\tilde f(\lambda)$ is $(\sigmamin^2/L_f)$-strongly convex. Combining the fact that $b^\T\lambda+h^\star(\lambda)$ is convex, we complete the proof. \qed
\end{proof}

One can observe that the Lipschitz continuity of $\nabla f(x)$ is transferred to the strong concavity of $\Phi(\lambda)$ through the matrix $A$. Therefore, a linear convergence can be expected. To exploit this observation, we perform an inexact proximal point algorithm to solve the dual problem \eqref{eqn:comp-dual}: 
\begin{equation}\label{eqn:SC-loop1}
\lambda_k \approx \argmax_\lambda\ \Phi_k(\lambda)\defeq \Phi(\lambda)-\frac{\ell}{2}\|\lambda-\lambda_{k-1}\|^2,
\end{equation}
where $\lambda_{k-1}$ is the iterate of the $(k-1)$-th step. Now it remains to solve the subproblem \eqref{eqn:SC-loop1} efficiently. By expanding the term $f^\star(-A^\T \lambda)$ through the conjugate function again, we can rewrite \eqref{eqn:SC-loop1} into the equivalent saddle point problem
\begin{equation}\label{eqn:SC-minimaxk}
\max_\lambda\min_x\ {\mathcal{L}}_k(x,\lambda)\defeq f(x)+\lambda^\T (Ax-b)-h^\star(\lambda)-\frac{\ell}{2} \|\lambda-\lambda_{k-1}\|^2.
\end{equation}
Let $g_k(\lambda)\defeq h^\star(\lambda)+\frac{\ell}{2}\|\lambda-\lambda_{k-1}\|^2$. Then the dual problem of \eqref{eqn:SC-loop1} is given by

\begin{equation}\label{eqn:SC-dual}
\min_x\ \Psi_k(x)\defeq f(x)+g_k^\star(Ax-b).
\end{equation}

The function $\Psi_k(\cdot)$ is $L_{\Psi}$-smooth and $\muf$-strongly convex, with $L_{\Psi}=L_f+\frac{\LA^2}{\ell}$. This time, the strong convexity induced by the proximal term for the dual variable $\lambda$ is transferred to the Lipschitz smoothness in the primal variable $x$.
With a few more computations, we show that $\nabla \Psi_k$ can be easily evaluated, and hence AGD can be applied to solve \eqref{eqn:SC-dual}.

\begin{proposition}
The gradient $\nabla \Psi_k(\cdot)$ can be evaluated with one call of $A$, one call of $A^\T$, one call of $\prox_{\ell h}(\cdot)$, and one call of $\nabla f(\cdot)$, respectively.
\end{proposition}
\begin{proof}
With $\lambda_k^\star(x)\defeq \argmax_\lambda \mathcal{L}_k(x,\lambda)$, Danskin's theorem indicates that
\[
    \nabla \Psi_k(x)=\nabla_x \mathcal{L}_k(x,\lambda_k^\star(x))=\nabla f(x)+A^\T\lambda_k^\star(x).
\]
In fact, $\lambda_k^\star(x)$ can be explicitly written as 
\begin{equation}
\label{eqn:lambdastar}
    \begin{split} \lambda_k^\star(x)&=\argmax_\lambda\ \left\{-h^\star(\lambda)-{\frac{\ell}{2}} \left\|\lambda-\lambda_{k-1}-\frac{Ax-b}{\ell}\right\|^2\right\}\\
    &=\prox_{\frac{h^\star}{\ell}}\left(\lambda_{k-1}+\frac{Ax-b}{\ell}\right)\\
    &=\frac1{\ell}\prox_{{\ell} h}\left({\ell}\lambda_{k-1}+Ax-b\right)-\lambda_{k-1}-\frac{Ax-b}{\ell},
    \end{split}
\end{equation}
where the last equality comes from Moreau's decomposition theorem \cite[Proposition IV.1.8]{showalter1997monotone}. \qed
\end{proof}

Therefore, $\nabla \Psi_k(x)$ can be easily evaluated and we can apply Algorithm \ref{algo:AGD} to efficiently obtain $x_k\approx\argmin_x\Psi_k(x)$ and then update $\lambda_k=\lambda_k^\star(x_k)$ by \eqref{eqn:lambdastar}, which will be proved to be an approximate solution of the subproblem \eqref{eqn:SC-loop1}. 

\subsection{Strongly convex case: Analysis}\label{sec:SC-ub-detail}

Now, we are ready to give the complete algorithm  for strongly convex problems. 
\begin{algorithm2e}[H]\label{algo:SC-APPA}
    \caption{Inexact PPA for strongly convex problems}
    \textbf{Input:} initial point $x_0,\lambda_0$, smoothness parameter $L_f$, strong convexity parameter $\mu_f$, minimal singular value $\sigmamin$, regularization parameter $\ell\geq\muphi$, radius parameter $D$.\\
 \textbf{Initialize:} $\rho=\frac{\muphi}{12\ell}, \delta_k=(1-\rho)^{\frac k2}D$. \\
	\For{$k=1,\dots, T$}
	{
        Update $x_k\leftarrow\AGD(\Psi_k, x_{k-1}, \delta_k)$. 

        Update $\lambda_k$ by \eqref{eqn:lambdastar}:
        $\lambda_k\leftarrow\prox_{\frac{h^\star}{\ell}}\left(\lambda_{k-1}+\frac{Ax_k-b}{\ell}\right).$
        
	}
    \textbf{Output:} $x_T$.
\end{algorithm2e}

\newcommand{\hkp}{\hat{\kappa}}
\newcommand{\hkphi}{\hkp_\Phi}
\newcommand{\rphi}{r_{\Phi}}
\newcommand{\reps}{r_0}

\begin{proposition}\label{prop:sc}
Suppose that $f(x)$ is $L_f$-smooth and $\mu_f$-strongly convex, the matrix $A$ satisfies $\nrm{A}\leq\LA$ and the minimum singular value of $A$ is no smaller than $\muA$. Let $(x^\star,\lambda^\star)$ be the pair of optimal primal and dual variables, and $\{(x_k,\lambda_k)\}$ be the iterate sequence generated by Algorithm \ref{algo:SC-APPA}. Assume that $\ell\geq\mu_\Phi$, then we have for any $0\leq k\leq T$, it holds that
\begin{equation}\label{eqn:lambda-linear}
    \nrm{\lambda_k-\lams}\leq (1-\rho)^{t/2}\cdot M,
\end{equation}
where $M=\max\set{\nrm{\lambda_0-\lams}, 10\LA\muphi^{-1}D}$. If we denote $x_k^\star=\argmin_x \Psi_k(x)$, then for any $2\leq k\leq T$, it holds that
\begin{align}
    \label{eqn:subD}\nrm{x_{k-1}-\xs_{k}}
    \leq&~ \frac{\LA}{\muf}\nrm{\lambda_{k-1}-\lambda_{k-2}}+\delta_{k-1},\\
    \label{eqn:x-linear}\nrm{x_k-\xs}
    \leq&~ \frac{\LA}{\muf}\nrm{\lambda_{k-1}-\lams}+\delta_k.
\end{align}
\end{proposition}

\begin{proof}
Let $\lams_k=\argmax_\lambda \Phi_k(\lambda)$ and $x_k^\star=\argmin_x \Psi_k(x)$. Recall that the output $x_k$ of Algorithm $\AGD(\Psi_k,x_{k-1},\delta_k)$ guarantees that $\nrm{x_k-\xs_k}\leq \delta_k$. 
According to the dual update rule \eqref{eqn:lambdastar} and the optimality of $(\xs_k,\lams_k)$, we have
\[
    \lambda_k^\star=\prox_{\frac{h^\star}{\ell}}\left(\lambda_{k-1}+\frac{Ax_k^\star-b}{\ell}\right),\quad
    \lambda_k=\prox_{\frac{h^\star}{\ell}}\left(\lambda_{k-1}+\frac{Ax_k-b}{\ell}\right).
\]
Thus, we can deduce $\nrm{\lambda_k-\lams_k}\leq \nrm{\frac{Ax_k-A\xs_k}{\ell}}\leq \frac{\LA \delta_k}{\ell}$ due to the non-expansiveness of proximal operator. By the definition of $\lams_k$, it holds that
\begin{align*}
    0\in \partial \Phi_k(\lams_k)=\partial \Phi(\lams_k) + \ell(\lams_k-\lambda_{k-1}),
\end{align*}
which implies that there exists $\Phi^\prime_k\in\partial \Phi(\lams_k)$ satisfying $\Phi^\prime_k + \ell(\lams_k-\lambda_{k-1})=0$.
Combining the fact that $\Phi$ is $\muphi$-strongly concave and $\lambda^\star=\argmax_\lambda\Phi(\lambda)$, we have
\begin{align*}
    \muphi\nrm{\lams_k-\lams}^2
    \leq&~ 
    \iprod{ \Phi^\prime_k }{ \lams_k-\lams }
    = \ell \iprod{ \lambda_{k-1}-\lams_k }{ \lams_k-\lams }\\
    =&~ 
    \frac{\ell}{2}\paren{ \nrm{\lambda_{k-1}-\lams}^2-\nrm{\lams_k-\lams}^2-\nrm{\lams_k-\lambda_{k-1}}^2 }.
\end{align*}
Hence, for any constant $c>0$, we have
\begin{align*}
    \frac{\ell}{\ell+2\muphi}\nrm{\lambda_{k-1}-\lams}^2 \geq \nrm{\lams_k-\lams}^2 \geq (1+c)^{-1}\nrm{\lambda_k-\lams}^2-c^{-1}\nrm{\lams_k-\lambda_k}^2.
\end{align*}
Picking $c=\frac{\muphi}{6\ell}$ and utilizing the assumption that $\ell\geq\mu_\Phi$ yields that
\begin{align*}
    \nrm{\lambda_k-\lams}^2\leq \paren{1-\frac{\muphi}{6\ell}}\nrm{\lambda_{k-1}-\lams}^2+\frac{7\ell}{\muphi}(C_1\delta_k)^2,
\end{align*}
where we denote $C_1=\LA/\ell$.
Therefore, by recursively applying the inequality above, we have
\begin{equation*}
    \nrm{\lambda_k-\lams}^2 \leq (1-2\rho)^{k}\nrm{\lambda_0-\lams}^2+\frac{7}{12\rho}\sum_{t=1}^{k} (1-2\rho)^{k-t}(C_1\delta_t)^2.
\end{equation*}
By our choice of $\delta_t=(1-\rho)^{t/2}D$, it holds
\begin{align*}
    \sum_{t=1}^{k} (1-2\rho)^{k-t}(\delta_t)^2\leq\brac{ (1-\rho)^k-(1-2\rho)^k }\cdot \frac{D^2}{\rho}.
\end{align*}
Putting these pieces together and plugging in the definition of $C_1$ and $\rho$, we get \eqref{eqn:lambda-linear}.

Recall the definition of $\Psi_k$ given in \eqref{eqn:SC-dual}. We can it rewrite it into
\begin{align*}
    \Psi_k(x)=f(x)+\hat{h}(Ax-b+\ell \lambda_{k-1})-\frac{\ell}{2}\|\lambda_{k-1}\|^2,
\end{align*}
where $\hat{h}$ is defined by
\begin{align*}
    \hat{h}(u)=\max_{\lambda}\paren{ 
\lambda^\T u - h^\star(\lambda)-\frac{\ell}{2}\nrm{\lambda}^2 }.
\end{align*}
Note that $\hat{h}(u)$ is the conjugate function of $h^\star(\lambda)+\frac\ell2\|\lambda\|^2$, and hence it follows from Lemma \ref{lemma:conjugate} that $\hat{h}(u)$ is $\frac{1}{\ell}$-smooth. By definition, we know $\nabla \Psi_k(\xs_k)=\nabla \Psi_{k-1}(\xs_{k-1})$, and hence
\begin{align*}
    \mu_f\nrm{\xs_k-\xs_{k-1}}
    \leq&~ \nrm{ \nabla \Psi_k(\xs_k) - \nabla \Psi_k(\xs_{k-1}) }
    = \nrm{ \nabla \Psi_{k-1}(\xs_{k-1}) - \nabla \Psi_k(\xs_{k-1}) } \\
    =&~ \nrm{ A^\T \nabla \hat{h}(A\xs_{k-1}-b+\ell\lambda_{k-2}) \!-\! A^\T \nabla \hat{h}(A\xs_{k-1}-b+\ell\lambda_{k-1}) }\\
    \leq&~ \nrmA \cdot \frac{1}{\ell} \cdot \ell \nrm{\lambda_{k-2}-\lambda_{k-1}}=\nrmA\nrm{\lambda_{k-2}-\lambda_{k-1}},
\end{align*}
where the first inequality follows from the $\muf$-strong-convexity of $f$, and the last inequality holds because $\hat{h}$ is $\frac{1}{\ell}$-smooth. Combining the fact that $\nrm{x_{k-1}-\xs_{k}}\leq \delta_k+\nrm{\xs_{k-1}-\xs_{k}}$, we get \eqref{eqn:subD}.

Similarly, by the optimality of $(\xs,\lams)$, we have $\xs$ is the minimum of
\begin{align*}
    \Psi_\star(x)\defeq f(x)+\hat{h}(Ax-b+\ell \lams)-\frac{\ell}{2}\|\lams\|^2.
\end{align*}
Repeating the argument above yields \eqref{eqn:x-linear}. \qed
\end{proof}

\begin{theorem}\label{thm:SC-upper}
Under the same assumption of Proposition \ref{prop:sc} and $\nrm{x_0-x^\star}\leq D$, in each iteration of AGD, the number of inner iteration $T_k$ of $\AGD(\Psi_k,x_{k-1},\delta_k)$ can be upper bound by
\begin{align*}
    T_k\leq 8\sqrt{\frac{L_f+\ell^{-1}\LA^2}{\muf}}\log\paren{\frac{10\kappa_f\kappa_A D_\star}{D}},
\end{align*}
where $D_\star\defeq \nrm{x_0-\xs}+\frac{\LA}{L_f}\nrm{\lambda_0-\lams}$. Furthermore, for Algorithm \ref{algo:SC-APPA} to find an approximate solution $x_T$ satisfying $\nrm{x_T-\xs}\leq \epsilon$, the number of outer iterations is
\begin{align*}
    T\leq \frac{12\ell}{\muphi}\log\paren{100\kappa_f\kappa_A\cdot\frac{D}{\epsilon}}.
\end{align*}
\end{theorem}
One comment here is that although we assume knowing an upper bound $D$ of the distance from $x_0$ to the optimal solution $x^*$, it only appears in the logarithmic terms. Hence one can use a very loose upper estimation of the distance in the algorithm without deteriorating the performance. 
\begin{proof}
We first consider the case $k\geq 2$.
By combining \eqref{eqn:subD} and \ref{eqn:lambda-linear}, we have
\begin{align*}
    \nrm{x_{k-1}-\xs_{k}}
    \leq \frac{\LA}{\muf}\cdot 2(1-\rho)^{k/2-1}M+\delta_{k-1}
    \leq (1-\rho)^{k/2}\paren{2D+\frac{3\LA}{\muf}M}.
\end{align*}
where the last inequality holds because $0<\rho\leq\frac1{12}$.
Therefore, by Proposition \ref{prop:AGD} and our choice $\delta_k=(1-\rho)^{k/2}D$, the inner number of steps $T_k$ of $\AGD(\Psi_k,x_{k-1},\delta_k)$ can be upper bound by
\begin{align*}
    T_k
    \leq&~
    2\sqrt{\frac{L_{\Psi}}{\mu_{\Psi}}}\log\paren{ \frac{2L_{\Psi}}{\mu_{\Psi}}\cdot \frac{\nrm{x_{k-1}-\xs_{k}}}{\delta_k} }\\
    \leq&~
    2\sqrt{\frac{L_{\Psi}}{\mu_{\Psi}}}\log\paren{ \frac{2L_{\Psi}}{\mu_{\Psi}}\cdot\paren{2+\frac{3\LA}{\muf}\frac{M}{D}} } \\
    \leq&~ 8\sqrt{\frac{L_f+\ell^{-1}\LA^2}{\muf}}\log\paren{10\kappa_f\kappa_A\left(1+\frac{\LA\nrm{\lambda_0-\lams}}{L_f D}\right)}.
\end{align*}
where the last inequality follows from plugging in the definition of $M$, $L_\Psi$ and $\mu_{\Psi}=\muf$. The case $k=1$ follows similarly.

By combining \eqref{eqn:x-linear} and \eqref{eqn:lambda-linear}, we have
\begin{align*}
    \nrm{x_T-\xs}\leq \frac{\LA}{\muf}\nrm{\lambda_{T-1}-\lams}+\delta_T
    \leq (1-\rho)^{T/2}\paren{D+\frac{3\LA}{\muf}M}.
\end{align*}
The desired result follows immediately. \qed
\end{proof}

\begin{corollary}\label{coro:strongly-convex}
Suppose that the same assumptions of Proposition \ref{prop:sc} hold and $D_\star\leq D$. In order to find an approximate solution $x_T$ satisfying $\suboptsc(x_T)\leq\epsilon$, the total number of gradient evaluations for Algorithm \ref{algo:SC-APPA} is bounded by
    \[
    \tilde O\left(\kappa_A\sqrt{\kappa_f}\log(D/\epsilon)\right).
    \]
\end{corollary}
\begin{proof}
According to Theorem \ref{thm:SC-upper}, if we let $\ell=\mu_\Phi$, then the complexity of each inner loop is $\tilde O(\kappa_A\sqrt{\kappa_f})$, the complexity of outer loop is $\tilde O\left(\log(D/\epsilon)\right)$. Therefore, the overall complexity is $\tilde O\left(\kappa_A\sqrt{\kappa_f}\log(D/\epsilon)\right)$.
\end{proof}

\subsection{Non-convex case}
For non-convex problems, our algorithm is present in Algorithm \ref{PPA-nonconvex}, which employs  the inexact proximal point algorithm in the outer iterations while solving the strongly convex subproblems via Algorithm \ref{algo:SC-APPA}. 

\begin{algorithm2e}[H]\label{PPA-nonconvex}
	\caption{Inexact PPA for non-convex problems}
     \textbf{Input:} initial point $x_0$, smoothness parameter $L_f$, condition number $\kappa_A$, subproblem error tolerance $\{\delta_k\}$ and the maximum iteration number $T$.\\
    \For{$k=1,\cdots, T$}
    {
        Apply Algorithm \ref{algo:SC-APPA} to find
        \begin{equation}\label{NC-PPA}
            x_{k}\approx x_{k}^\star:=\argmin_x \left\{F(x)+L_f\|x-x_{k-1}\|^2\right\},
        \end{equation}
        such that 
        $\|x_{k}-x_{k}^\star\|\leq\delta_{k}$.
    }
    \textbf{Output:} $\{x_k\}_{k=1}^T$.
\end{algorithm2e}

Under a suitable sequence of $\{\delta_k\}$, we provide the convergence rate of Algorithm \ref{PPA-nonconvex} in the following theorem. 

\begin{theorem}\label{thm:NC-PPA}
    Suppose that $f(x)$ is $L_f$-smooth, the condition number of $A$ is $\kappa_A$. 
    Assume that $F(x_1^\star)-\inf_x F(x)\leq\Delta^\prime$, $\delta_k=\frac{\sqrt{\Delta^\prime/L_f}}{2k}$ and $\{x_k\}$ be the iterate sequence generated by Algorithm \ref{PPA-nonconvex}. Then we have 
    \[
        \min_{0\leq k<T} \mathrm{SubOpt}(x_k)\leq \sqrt{\frac{5L_f\Delta^\prime}{T}}.
    \]
    In particular, suppose that the initial point $x_0\in\dom F$ and $F(x_0)-\inf_x F(x)\leq\Delta$ for some $\Delta>0$, then taking $\Delta'=\Delta$ yields
    \begin{equation}\label{eqn:NC-upper}
    \min_{0\leq k<T} \mathrm{SubOpt}(x_k)\leq \sqrt{\frac{5L_f\Delta}{T}}.
    \end{equation}
\end{theorem}
\begin{proof}
By the definition of $x_{k+1}^\star$, on the one hand, we have $F(x_{k+1}^\star)+L_f\|x_{k+1}^\star-x_k\|^2\leq F(x_k^\star)+L_f\|x_k^\star-x_k\|^2$, which yields
\begin{equation}\label{eqn:NC-star1}
\|x_{k+1}^\star-x_k\|^2\leq \frac1{L_f}(F(x_k^\star)-F(x_{k+1}^\star))+\delta_k^2.
\end{equation}
On the other hand, it holds
\[
x_{k+1}^\star=\proxhAx\left(x_k-\frac1{2L_f}\nabla f(x_{k+1}^\star)\right).
\]
Therefore, we have
\begin{align*}
&\left\|x_k-\proxhAx\left(x_k-\frac1{2L_f}\nabla f(x_k)\right)\right\|\\
\leq & \ \|x_k-x_{k+1}^\star\|+\left\|x_{k+1}^\star-\proxhAx\left(x_k-\frac1{2L_f}\nabla f(x_k)\right)\right\|\\
\leq & \ 2\|x_k-x_{k+1}^\star\|,
\end{align*}
where the last inequality holds because the proximal operator is non-expansive and $f(\cdot)$ is $L_f$-smooth. Combining with \eqref{eqn:NC-star1}, we obtain
\begin{align*}
&L_f^2\left\|x_k-\proxhAx\left(x_k-\frac1{2L_f}\nabla f(x_k)\right)\right\|^2\\
&\leq 4L_f^2\|x_k-x_{k+1}^\star\|^2\leq 4L_f(F(x_k^\star)-F(x_{k+1}^\star))+4L_f^2\delta_k^2.
\end{align*}
Summing up the above inequality over $t = 1, \cdots, T$,
\begin{align*}
    \frac 1T \sum_{k=1}^{T}L_f^2\left\|x_{k}\!-\!\proxhAx\left(x_{k}\!-\!\frac{\nabla f(x_{k})}{2L_f}\right)\right\|^2
    \!\leq\! \frac{4L_f(F(x_1^\star)\!-\!F(x_{T+1}^\star))\!+\! L_f\Delta^\prime}{T}.
\end{align*}
Since $F(x_1^\star)-F(x_{T+1}^\star)\leq F(x_1^\star)-\inf_x F(x)\leq \Delta^\prime$, it holds that
\[
    \min_{1\leq k\leq T} L_f\left\|x_{k}-\proxhAx\left(x_{k}-\frac{\nabla f(x_{k})}{2L_f}\right)\right\|\leq \sqrt{\frac{5L_f\Delta^\prime}{T}}.
\]

Furthermore, under the assumption of $x_0\in\dom F$, it holds $F(x_1^\star)\leq F(x_0)$ and accordingly $F(x_1^\star)-\inf_x F(x)\leq F(x_0)-\inf_x F(x)\leq\Delta$. This completes the proof. 
\qed
\end{proof}

In Theorem \ref{thm:NC-PPA}, we present the first inequality by incorporating the definition of $\Delta^\prime$ specifically for the scenario where $x_0\notin \dom F$. This is necessary because in such cases, $F(x_0)-\inf_x F(x)$ can be infinite, as exemplified by $h(x)$ when it is an indicator function. Consequently, it becomes impossible to find a finite $\Delta$ that satisfies $F(x_0)-\inf_x F(x)\leq \Delta$. By introducing the definition of $\Delta^\prime$, we ensure the existence of a finite $\Delta^\prime$ and thus establish a well-defined result.

\begin{corollary}
    Under the same assumption and same choice of $\delta_k$ in Theorem \ref{thm:NC-PPA}, in order to find an approximate solution $x_T$ satisfying $\suboptnc(x_T)\leq \epsilon$, the total number of gradient evaluations for Algorithm \ref{PPA-nonconvex} is bounded by
    \[
        \tilde{{O}}\left(\frac{\kappa_A L_f\Delta^\prime}{\epsilon^2}\right).
    \]
    Furthermore, if $x_0\in\dom F$, the total number is bounded by $\tilde{{O}}\left(\frac{\kappa_A L_f\Delta}{\epsilon^2}\right)$.
\end{corollary}
\begin{proof}
    By Theorem \ref{thm:NC-PPA}, to reach the expected precision, we need $O\left(\frac{L_f\Delta^\prime}{\epsilon^2}\right)$ outer iterations.
    For each $1\leq k\leq T$, the function $f(x)+L_f\|x-x_{k-1}\|^2$ is $2L_f$-strongly-convex and $3L_f$-smooth, and hence its condition number is $O(1)$.
    Hence, the number of gradient evaluations in the $k$-th inner iteration with Algorithm \ref{algo:SC-APPA} is $\tilde O(\kappa_A\log(1/\delta_k))$ by Corollary \ref{coro:strongly-convex}.
    Combining the complexities of inner and outer loops, we obtain the $\tilde{{O}}\left(\frac{\kappa_A L_f\Delta^\prime}{\epsilon^2}\right)$ overall complexity. For the case $x_0\in\dom F$, the complexity can be derived similarly.\qed
\end{proof}

\subsection{Convex case}\label{sec:C-ub}
For any given $x_0$ and $\epsilon>0$, we construct the following auxiliary problem:
\begin{equation}\label{eqn:C-obj}
\min_x\ f(x)+h(Ax-b)+\frac{\epsilon}{2D^2}\|x-x_0\|^2.
\end{equation}
The smooth part $f(x)+\frac{\epsilon}{2D^2}\|x-x_0\|^2$ is strongly convex and hence we can apply Algorithm \ref{algo:SC-APPA} to solve the problem. The following corollary illustrates that the approximate solution of \eqref{eqn:C-obj} is also an approximate solution of the original convex problem and the overall complexity is also optimal.

\begin{corollary}\label{coro:C-upper}
    Suppose that $f(x)$ is convex, the condition number of $A$ is $\kappa_A$ and $\|x_0-x^\star\|\leq D$. For any given $0<\rho<+\infty$, Algorithm \ref{algo:SC-APPA} can be applied on problem \eqref{eqn:C-obj} and output an approximate solution $\hat x$ satisfying $\suboptc(\hat x)\leq\epsilon$, such that the total number of gradient evaluations is bounded by
    \[
    \tilde O\left(\frac{\kappa_A\sqrt{L_f}D}{\sqrt{\epsilon}}\right),
    \]
    where $\tilde O$ also hides the logarithmic dependence on $\rho$.
\end{corollary}
\begin{proof}
\newcommand{\xeps}{x_\epsilon^\star}
Denote the exact solution of \eqref{eqn:C-obj} as $x_\epsilon^\star$. 
We apply Algorithm \ref{algo:SC-APPA} on \eqref{eqn:C-obj} and calculate a point $\hx$ such that $\nrm{\hx-\xeps}\leq\delta$, where we will specify $\delta$ later in proof.

By the optimality of $\xeps$, it holds that
\begin{align*}
    F(\xeps)+\frac{\epsilon}{2D^2}\nrm{\xeps-x_0}^2
    \leq
    F(\xs)+\frac{\epsilon}{2D^2}\nrm{\xs-x_0}^2.
\end{align*}
In particular, we have $F(\xeps)\leq F(\xs)+\epsilon/2$, and by the fact that $F(\xs)\leq F(\xeps)$, we know $\nrm{\xeps-x_0}\leq \nrm{\xs-x_0}\leq D$.

On the other hand, we have
\begin{align*}
    f(\hx)
    \leq&~ f(\xeps)+\iprod{\nabla f(\xeps)}{\hx-\xeps}+\frac{L_f}{2}\nrm{\hx-\xeps}^2\\
    \leq&~ f(\xeps)+\paren{\nrm{\nabla f(x_0)}+L_fD}\nrm{\hx-\xeps}+\frac{L_f}{2}\nrm{\hx-\xeps}^2,
\end{align*}
where the last inequality follows from $\nrm{\nabla f(\xeps)-\nabla f(x_0)}\leq L_f\nrm{\xeps-x_0}\leq L_fD$.
Further, since $h_\rho(\cdot)$ is $\rho$-Lipschitz continuous and $h_\rho(\cdot)\leq h(\cdot)$, it holds
\[
h_\rho(A\hx-b)\leq h_{\rho}(A\xeps-b)+\rho\nrm{A\hx-A\xeps} 
\leq h(A\xeps-b)+\rho\LA\nrm{\hx-\xeps}.
\]
Denote $C_{\rho}=\nrm{\nabla f(x_0)}+L_fD+\rho \LA$. Combining the above two inequalities yields
\begin{align*}
f(\hx)+h_\rho(A\hx-b)
\leq &~
f(\xeps)+h(A\xeps-b)+C_{\rho}\nrm{\hx-\xeps}+\frac{L_f}{2}\nrm{\hx-\xeps}^2\\
\leq &~
F(\xeps)+C_\rho\delta+\frac{L_f\delta^2}{2}\\
\leq &~
F(\xs)+\frac{\epsilon}{2}+C_\rho\delta+\frac{L_f\delta^2}{2}.
\end{align*}
Therefore, we can set
\begin{align*}
    \delta=\min\set{\frac{\epsilon}{3C_\rho},\sqrt{\frac{\epsilon}{3L_f}}}
\end{align*}
to ensure that $\suboptc(\hx)\leq \epsilon$.
Notice that the function $f(x)+\frac{\epsilon}{2D^2}\nrm{x-x_0}^2$ is $(L_f+\epsilon/D^2)$-smooth and $(\epsilon/D^2)$-strongly convex. Therefore, according to Corollary \ref{coro:strongly-convex}, the required number of gradient evaluations is $\tilde O\left(\kappa_A\sqrt{\frac{L_fD^2}{\epsilon}+1}\right)$.
\qed
\end{proof}

Note that we give some specific definitions of $h_\rho$ in Property \ref{hrho-cone} and \ref{hrho-Lip}. In the following, we give the complexity results on these specific problems.
\begin{corollary}\label{cor:conic-comp}
For the conic inequality constrained problem \eqref{conic-constrained} and any fixed $\rho>0$, in order to find an approximate solution $x_T$ satisfies
\begin{align*}
    |f(x_T)-f(x^\star)|\leq\epsilon, \qquad
    \|\mathcal{P}_{\mathcal{K}^\circ}(Ax_T-b)\|\leq \frac{\epsilon}{\rho},
\end{align*}
the required number of gradient evaluation is $\tilde O\left(\kappa_AD\sqrt{L_f/\epsilon}\right)$. 
\end{corollary}

Corollary \ref{cor:conic-comp} implies that for conic inequality constrained convex problems (including linearly constrained convex problems), the constraint can be fulfilled to arbitrary accuracy without affecting the order of the complexity (up to log factor).

\begin{corollary}
When $h(\cdot)$ is $\rho$-Lipschitz continuous (e.g., the norm regularized problem \eqref{regularized-prob}), in order to find an approximate solution $x_T$ satisfies $F(x_T)-\min_x F(x)\leq\epsilon$, the required number of gradient evaluation is $\tilde O\left(\kappa_AD\sqrt{L_f/\epsilon}\right)$.
\end{corollary}

\section{Lower bounds}\label{sec:lower-bound}

\subsection{Problem classes and algorithm class}\label{sec:settings}
In this section, we aim to construct three hard instances for the strongly convex, convex, and non-convex cases, respectively. First, let us formally define the problem classes and the linear span first-order algorithm class. For the simplicity of presentation, we construct the hard instances with $x_0=0$ and $L_A=2$.\vspace{0.2cm}

\noindent\textbf{Strongly convex problem class.}\, For positive constants $L_f\geq\mu_f>0$, $D,\kappa_A>0$, the problem class $\mathcal{F}_{\mathrm{SC}}(L_f,\mu_f,\kappa_A,D)$ includes problems in which $f(x)$ is $L_f$-smooth and $\mu_f$-strongly convex, $\|x_0-x^\star\|\leq D$, 
and the condition number of $A$ is upper bounded by $\kappa_A$.\vspace{0.2cm} 

\noindent\textbf{Non-convex problem class.}\, For positive constants $L_f,\kappa_A,\Delta>0$ and $x_0\in\dom F$, the problem class $\mathcal{F}_{\mathrm{NC}}(L_f,\Delta,\kappa_A)$ includes problems where $f(x)$ is $L_f$-smooth, $F(x_0)-F(x^\star)\leq\Delta$ and the condition number of $A$ is upper bounded by $\kappa_A$. \vspace{0.2cm}

\noindent\textbf{Convex problem class.}\, For positive constants $L_f,D,\kappa_A>0$, the problem class $\mathcal{F}_{\mathrm{C}} (L_f,\kappa_A, D)$ includes problems in which $f(x)$ is $L_f$-smooth, $\|x_0-x^\star\|\leq D$, and the condition number of $A$ is upper bounded by $\kappa_A$. \vspace{0.2cm}

For the above three problem classes, we restrict our discussion to first-order linear span algorithms. The results can be extended to the general first-order algorithms without linear span structure by the orthogonal invariance trick proposed in \cite{carmon2020lower}.\vspace{0.2cm}

\noindent\textbf{First-order linear span algorithms.}\, The iterate sequence $\{(x_k,\lambda_k)\}$ is generated such that $(x_k,\lambda_k)\in\mathcal{S}^x_{k+1}\times\mathcal{S}^\lambda_{k+1}$. These subspaces are generated by starting with $\mathcal{S}^x_0=\mathrm{Span}\{x_0\}$, $\mathcal{S}^\lambda_0=\mathrm{Span}\{\lambda_0\}$ and
\begin{equation*}
    \begin{split}
     \mathcal{S}^x_{k+1} & \!\defeq\!\mathrm{Span}\left\{x_i,\nabla f(\hat x_i), A^\T \hat\lambda_i: \forall \hat x_i\in\mathcal{S}_i^x, \hat\lambda_i\in\mathcal{S}_i^\lambda, 0\leq i\leq k\right\},\\
    \mathcal{S}^\lambda_{k+1}& \!\defeq\!\mathrm{Span}\left\{\lambda_i,\prox_{\eta_i h^\star}\left(\hat \lambda_i + \eta_i(A\hat x_j-b)\right): \forall \hat x_i\in\mathcal{S}_i^x, \hat\lambda_i\in\mathcal{S}_i^\lambda, 0\leq i\leq k\right\}.
    \end{split}
\end{equation*}
If we assume that $\lambda_0=0$, then for the linear equality constrained problem \eqref{linear-equality}, $\prox_{\eta_i h^\star}(x)=x$ and the algorithm class degenerates into
\[ 	
    x_{k+1}\in\mathcal{S}^x_{k+1}=\mathrm{Span}\left\{x_i,\nabla f(\hat x_i), A^\T(A\hat x_i-b), ~\forall \hat x_i\in\mathcal{S}_i^x, 0\leq i\leq k\right\}.
\]
We can further assume $x_0=0$ without loss of generality, otherwise, we can consider the shifted problem $\min_{x} F(x-x_0)$.

Note that for a first-order linear span algorithm, it is not necessary to use the current gradient in each iteration. Instead, it can use any combination of points from the historical search space. This makes the algorithm class general enough to cover diverse iteration schemes. To give some specific examples, we present following single-loop and double-loop algorithms covered under the considered algorithm class. 
\begin{example}[Single-loop algorithms]
Consider problem \eqref{linear-equality} with $h(\cdot)=\mathbb{I}_{\{0\}}(\cdot)$, the Chambolle-Pock method \cite{chambolle2011first}, the OGDA method \cite{mokhtari2020convergence} and the linearized ALM \cite{xu2021first} update iterates by the following rules
\begin{equation}\tag{Chambolle-Pock}
\begin{split}
    \left\{\begin{array}{l}
    x_{k+1}=x_k-\eta_1 \left(\nabla f(x_k)+A^\T \lambda_k\right)\\
    \lambda_{k+1}=\lambda_k+\eta_2 (2Ax_{k+1}-Ax_k-b)
    \end{array}\right.
\end{split}
\end{equation}
\begin{equation}\tag{OGDA}
\begin{split}
    \left\{\begin{array}{l}
    x_{k+1}=x_k-\eta_1 \left(2\nabla f(x_k)-\nabla f(x_{k-1})+A^\T (2\lambda_k-\lambda_{k-1})\right)\\
    \lambda_{k+1}=\lambda_k+\eta_2 (2Ax_k-Ax_{k-1}-b)
    \end{array}\right.
\end{split}
\end{equation}
\begin{equation}\tag{Linearized ALM}
\begin{split}
    \left\{\begin{array}{l}
    x_{k+1}=x_k-\eta_1 \left(\nabla f(x_k)+A^\T \lambda_k+\rho A^\T(Ax_{k}-b)\right)\\
    \lambda_{k+1}=\lambda_k+\eta_2 (Ax_{k+1}-b)
    \end{array}\right.
\end{split}
\end{equation}
where $\rho>0$ is a penalty factor.
\end{example}
For problem class $\mathcal{F}_{SC}(L_f,\mu_f,\kappa_A)$, a unified analysis on the above three methods was provided in \cite{zhu2022unified}, and an $O\left((\kappa_f+\kappa_A^2)\log\left(\frac1\epsilon\right)\right)$ complexity is achieved.
\begin{example}[Double-loop algorithm \cite{xu2021iteration}] Consider problem \eqref{linear-equality} with $h(\cdot)=\mathbb{I}_{\{0\}}(\cdot)$. Let $\mathcal{L}_\rho(x,\lambda)=f(x)+\lambda^\T(Ax-b)+\frac\rho2\|Ax-b\|^2$ be the augmented Lagrangian function. ALM generates iterates by
\[
\left\{\begin{array}{l}
x_{k+1}\approx\argmin_x \mathcal{L}_\rho(x,\lambda_k) \\
\lambda_{k+1}=\lambda_k+\rho (Ax_{k+1}-b)
\end{array}\right.
\]
where the subproblem is solved by an inner loop of Nesterov's AGD method.
An $O(\epsilon^{-1})$ complexity for convex problems and an $O(\epsilon^{-\frac12})$ complexity for strongly convex problems are derived in \cite{xu2021iteration}.
\end{example}
\begin{remark}
It can be checked that all three algorithms proposed in the upper bound section belong to the defined first order linear span algorithm class for general $h$.
\end{remark}

\subsection{The construction of hard instance}

In this section, we construct the hard instances for any first-order linear span algorithms. Specifically, we consider the linear equality constrained problem \eqref{linear-equality} with $h(\cdot)=\mathbb{I}_{\{0\}}(\cdot)$. For positive integers $N$ and $d$, we define the following problem
\begin{align}\label{eqn:OBJ}
\begin{aligned}
\min_{x\in\RR^{2Nd}}&~ f_0(x)\defeq G(x[1],x[N+1])+\cdots+G(x[N],x[2N]), \\
\st&~ x[1]=x[2]=\cdots=x[2N],
\end{aligned}
\end{align}
where $x[1],\cdots,x[2N]\in\RR^d$ and $x\in\RR^{2Nd}$ is the vector that stacks $x[i]$ together in order, the component function $G(u,v):\RR^{d}\times\RR^{d}\mapsto\RR$ is a smooth function to be determined later.
To ensure that $f_0(x)$ satisfies the assumptions of different problem classes, we will construct various formulations of $G(u,v)$ in the strongly convex, convex, and non-convex cases, respectively. Additionally, we require  $G(u,v)$ to satisfy the following assumption.

\begin{assumption}\label{zero-chain}
For any $i\geq 0$, it holds
\begin{enumerate}
    \item If $\supp\left\{u\right\}\subset[i+1], \supp\left\{v\right\}\subset[i]$, then $\supp \left\{\nabla_u G(u,v)\right\}\subset [i+1]$ and $\supp\left\{\nabla_v G(u,v)\right\}\subset [i]$.
    \item If $\supp\left\{u\right\}\subset[i], \supp\left\{v\right\}\subset[i]$, then $\supp\left\{\nabla_u G(u,v)\right\}\subset [i+1]$ and $\supp\left\{\nabla_v G(u,v)\right\}\subset [i]$.
\end{enumerate}
\end{assumption}
The constraint in \eqref{eqn:OBJ} can be rewritten as $Ax=0$ with
    \begin{equation}\label{def:A}
    A=\left[\begin{array}{ccccc}
    I_d & -I_d & & & \\
    & \ddots & \ddots & & \\
    & & I_d & -I_d & \\
    & & & I_d & -I_d
    \end{array}\right]\in \mathbb{R}^{(2N-1)d \times 2Nd}.
    \end{equation}
Hence $AA^\T\in \mathbb{R}^{(2N-1)d \times (2N-1)d}$ and $A^\T A\in \mathbb{R}^{2Nd \times 2Nd}$ can be computed as
    \begin{equation}
        AA^\T=\left[\begin{array}{ccccc}
        2I_d & -I_d & & & \\
        -I_d & 2I_d & -I_d & & \\
        & \ddots & \ddots & \ddots & \\
        & & -I_d & 2I_d & -I_d \\
        & & & -I_d & 2I_d
        \end{array}\right],
    ~ A^\T A = \left[\begin{array}{ccccc}
        I_d & -I_d & & & \\
        -I_d & 2I_d & -I_d & & \\
        & \ddots & \ddots & \ddots & \\
        & & -I_d & 2I_d & -I_d \\
        & & & -I_d & I_d
        \end{array}\right].
    \end{equation}

\begin{lemma}\label{lemma:kappaA}
    For matrix $A$ defined in \eqref{def:A}, its condition number satisfies $\kappa_A\!\leq\! \sqrt{2N^2\!-\!1}$.
\end{lemma}
\begin{proof}
Note that $T = AA^\T$ is a block tridiagonal Toeplitz matrix and its eigenvalue is $2+2\cos\left(\frac{\pi i}{2N}\right), 1\leq i\leq 2N-1$ (see \cite{noschese2013tridiagonal}).
Accordingly, the condition number of $T$ satisfies
\[
    \kappa_T = \frac{2+2\cos\left(\frac{\pi}{2N}\right)}{2+2\cos\left(\frac{\pi(2N-1)}{2N}\right)} = \frac{1+\cos\left(\frac{\pi}{2N}\right)}{1-\cos\left(\frac{\pi}{2N}\right)} = 1+\frac{2}{\frac{1}{\cos\left(\frac{\pi}{2N}\right)}-1}\leq 2N^2-1,
\]
where the last inequality is due to $\cos\left(\frac{\pi x} 2 \right)\leq1-x^2$, $\forall x\in\left[0,1\right]$.
Consequently, $\kappa_A\leq \sqrt{2N^2-1}$. \qed
\end{proof}

Next, we demonstrate the propagation of non-zero entries in this example. 
For first-order linear span algorithm with $b=0$, we have $$\mathcal{S}_{k+1}\subset \spa\left\{\mathcal{S}_{k}\cup \left\{\nabla f_0(\hat x_k), A^\T A \hat x_k\right\}\right\}.$$
It implies that new non-zero entries are introduced either through $\nabla G(x[i],x[N+i])$, or through the action of $A^\T A$ on $x$. 
As $A^\T A$ is a block tridiagonal matrix, each action of $A^\T A$ enables entries in $x[i]$ to "communicate" with their neighboring vectors, thereby propagating the non-zero entries.

Figure \ref{nozero-entry} illustrates how non-zero entries propagate. 
Assume that the initial point is $(x[i])_j=0$ for all $i$ and $j$.
In the first iteration, we use Assumption \ref{zero-chain} to observe that $\supp \left\{\nabla_u G(x[j],x[N+j])\right\}\subset [1], 1\leq j\leq N$ and $\supp \left\{\nabla_v G(x[j],x[N+j])\right\}=\emptyset$, so it is only possible to have $(x[1:N])_1\neq 0$. 
In the second iteration, $\nabla G(x[i],x[N+i])$ does not introduce any new non-zero entries, but the action of $A^\T A$ on $x$ causes $(x[N+1])_1$ to receive a non-zero entry from $(x[N])_1$. 
In the third iteration, we have $\nabla_u G(x[1],x[N+1])\subset [2]$, which allows $(x[1])_2$ to become non-zero. 
Additionally, $A^\T A$ propagates the non-zero entry in $(x[N+1])_1$ to $(x[N+2])_1$.
By repeating the above propagation mechanism, we can see that by the $(N+1)$th iteration, both $(x[1:2N])_1$ and $(x[1:N-1])_2$ become nonzero.
In the $(N+2)$th iteration, $(x[N])_2$ becomes nonzero through $\nabla_u G(x[N],x[2N])\subset [2]$.
We can consider iterations $2$ to $N+2$ (which consist of $N+1$ iterations) as one complete round of iterations. 
By repeating this process, each round of iteration can convert up to $2N$ elements to nonzero.
After $i-2$ rounds of iteration, specifically at the $((i-1)(N+1)+1)$th iteration, $(x[1:2N])_{1:i-1}$ and $(x[1:N])_i$ become possibly nonzero.

\begin{figure}[!ht]
\centering
\begin{tikzpicture}[
    roundnode/.style={circle, draw=green!60, fill=green!5, very thick, minimum size=7mm},
    squarednode/.style={rectangle, draw=gray!60, fill=gray!5, very thick, minimum size=5mm},
    ]
    \node[squarednode]      (1_1)                              {$(1,1)$};
    \node[squarednode]      (2_1)       [right=1cm of 1_1]     {$(2,1)$};
    \node[squarednode]      (N_1)       [right=1.5cm of 2_1]     {$(N,1)$};
    \node[squarednode]      (N+1_1)     [right=1cm of N_1]     {($N$+1, 1)};
    \node[squarednode]      (N+2_1)     [right=1cm of N+1_1]   {($N$+2, 1)};
    \node[squarednode]      (2N_1)      [right=1.5cm of N+2_1]   {($2N$, 1)};

    \node[squarednode]      (1_2)       [below=2cm of 1_1]     {$(1,2)$};
    \node[squarednode]      (2_2)       [below=2cm of 2_1]     {$(2,2)$};
    \node[squarednode]      (N_2)       [below=2cm of N_1]     {$(N,2)$};
    \node[squarednode]      (N+1_2)     [below=2cm of N+1_1]   {($N$+1, 2)};
    \node[squarednode]      (N+2_2)     [below=2cm of N+2_1]   {($N$+2, 2)};
    \node[squarednode]      (2N_2)      [below=2cm of 2N_1]    {($2N$, 2)};

    \node[squarednode]      (1_3)       [below=2cm of 1_2]     {$(1,3)$};
    \node[squarednode]      (2_3)       [below=2cm of 2_2]     {$(2,3)$};
    \node[squarednode]      (N_3)       [below=2cm of N_2]     {$(N,3)$};
    \node[squarednode]      (N+1_3)     [below=2cm of N+1_2]   {($N$+1, 3)};
    \node[squarednode]      (N+2_3)     [below=2cm of N+2_2]   {($N$+2, 3)};
    \node[squarednode]      (2N_3)      [below=2cm of 2N_2]    {($2N$, 3)};

    \node[right=0.5cm of 2_1]   (dots) {$\cdots$};
    \node[above right =0.6cm and 0.9cm of 1_1]   (11){};
    \node[above right =0.6cm and 0.9cm of 2_1]   (21){};
    \node[above right =0.6cm and 0.9cm of N_1]   (N1){};
    \draw[thick,blue!60,->]             (11.south) -- node [midway,sloped,above]{$1$} (1_1.north);
    \draw[thick,blue!60,->]             (21.south) -- node [midway,sloped,above]{$1$} (2_1.north);
    \draw[thick,blue!60,->]             (N1.south) -- node [midway,sloped,above]{$1$} (N_1.north);

    \draw[thick,orange,->]             (N_1.east) -- node [midway,above]{$2$} (N+1_1.west);
    \draw[thick,orange,->]             (N+1_1.east) -- node [midway,above]{$3$} (N+2_1.west);
    \draw[thick,orange,dashed,->]      (N+2_1.east) -- (2N_1.west);

    \draw[thick,orange,->]             (1_2.east) -- node [midway,above]{$4$} (2_2.west);
    \draw[thick,orange,dashed,->]      (2_2.east) -- (N_2.west);
    \draw[thick,orange,->]             (N_2.east) -- node [midway,above]{$N$+$3$} (N+1_2.west);
    \draw[thick,orange,->]             (N+1_2.east) -- node [midway,above]{$N$+$4$} (N+2_2.west);
    \draw[thick,orange,dashed,->]      (N+2_2.east) -- (2N_2.west);

    \draw[thick,orange,->]             (1_3.east) -- node [midway,above]{$N$+$5$} (2_3.west);
    \draw[thick,orange,dashed,->]      (2_3.east) -- (N_3.west);
    \draw[thick,orange,->]             (N_3.east) -- node [midway,above]{$2N$+$4$} (N+1_3.west);
    \draw[thick,orange,->]             (N+1_3.east) -- node [midway,above]{$2N$+$5$} (N+2_3.west);
    \draw[thick,orange,dashed,->]      (N+2_3.east) -- (2N_3.west);

    \draw[thick,blue!60,->]             (N+1_1.south) -- node [midway,sloped,above]{$3$} (1_2.north);
    \draw[thick,blue!60,->]             (N+2_1.south) -- node [midway,sloped,above]{$4$} (2_2.north);
    \draw[thick,blue!60,->]             (2N_1.south) -- node [midway,sloped,above]{$N$+$2$} (N_2.north);

    \draw[thick,blue!60,->]             (N+1_2.south) -- node [midway,sloped,above]{$N$+$4$} (1_3.north);
    \draw[thick,blue!60,->]             (N+2_2.south) -- node [midway,sloped,above]{$N$+$5$} (2_3.north);
    \draw[thick,blue!60,->]             (2N_2.south) -- node [midway,sloped,above]{$2N$+$3$} (N_3.north);

    \node[below left =1cm and 1.8cm of N+1_3]   (N+13){};
    \node[below left =1cm and 1.8cm of N+2_3]   (N+23){};
    \node[below left =1cm and 1.8cm of 2N_3]    (2N3){};
    \draw[thick,blue!60,->]             (N+1_3.south) -- node [midway,sloped,above]{$2N$+$5$} (N+13.north);
    \draw[thick,blue!60,->]             (N+2_3.south) -- node [midway,sloped,above]{$2N$+$6$} (N+23.north);
    \draw[thick,blue!60,->]             (2N_3.south) -- node [midway,sloped,above]{$3N$+$4$} (2N3.north);

    \node[below =1.5cm of 1_3, rotate=90]   (dots13){$\cdots$};
    \node[below =1.5cm of 2_3, rotate=90]   (dots23){$\cdots$};
    \node[below =1.5cm of N_3, rotate=90]   (dotsN3){$\cdots$};
    \node[below =1.5cm of N+1_3, rotate=90]   (dotsN+13){$\cdots$};
    \node[below =1.5cm of N+2_3, rotate=90]   (dotsN+23){$\cdots$};
    \node[below =1.5cm of 2N_3, rotate=90]   (dots2N3){$\cdots$};
\end{tikzpicture}
\caption{Propagation of nonzero entries. 
In this figure, the pair $(i,j)$ represents the entry $(x[i])_j$. 
The propagation is indicated by blue arrows when passing through $\nabla f_0(x)$, 
and by orange arrows when passing through $A^\T Ax$. 
The number of iteration is placed above the arrows.}
\label{nozero-entry}
\end{figure}
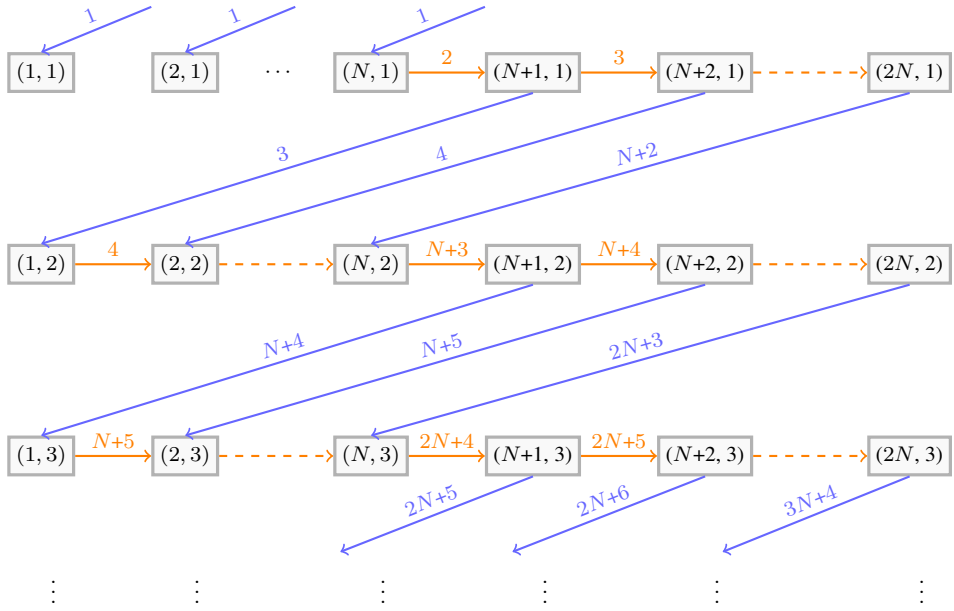

Based on the above procedure, it is natural to obtain the following Lemma.

\begin{lemma}\label{lemma:index}
    For $k=(i-1)(N+1)+j$ with $1\leq i\leq d-1, 1\leq j\leq N+1$, we have
\begin{align*}
    \supp\left\{x_k\right\}\subset \left\{(1:2N,1:i-1)\cup (1:N+j-1, i) \cup (1:j-2, i+1)\right\},
\end{align*}
    where the pair $(i,j)$ represents the entry $(x[i])_j$.
    Therefore, for any $k>0$, let $K=\left\lfloor \frac{k-2}{N+1} \right\rfloor+1$, then $(x[i])_j = 0$ for any $i,j$ satisfy $N+1\leq i\leq2N, K+1\leq j\leq d$.
\end{lemma}

It is well known that the complexity lower bound of the strongly convex case for an unconstrained smooth problem is $\Omega(\sqrt{\kappa_f}\log(1/\epsilon))$ \cite{nesterov2018lectures}, the lower bound of the convex case is $\Omega\left(\sqrt{L_f}D/\sqrt{\epsilon}\right)$ \cite{nesterov2018lectures} and the lower bound of non-convex case is $\Omega(L_f\Delta/\epsilon^2)$ \cite{carmon2020lower}.
These papers construct hard instances with zero-chaining property, where only one possible non-zero entry is added to the decision variable at each iteration (see Chapter 2 of \cite{nesterov2018lectures}). In contrast, Lemma \ref{lemma:index} indicates that in order to add a non-zero entry to each $x[i]$, it is necessary to take at least $N+1$ iterations, that is, $\Omega(\kappa_A)$ iterations by Lemma \ref{lemma:kappaA}. Therefore, our complexity lower bounds need to be multiplied by an additional factor of $\kappa_A$ on top of the lower bounds of unconstrained problems. This intuitively yields our results listed in Table \ref{table:ours}.

\subsection{Strongly convex case}\label{sec:SCLB}
Let us construct our hard problem instance based on formulation \eqref{eqn:OBJ}.
For any given positive parameters $L_f\geq\mu_f>0,\alpha>0$ and positive integers $d$, we define function $G(\cdot,\cdot):\RR^d\times \RR^d\mapsto\RR$ as
\begin{equation}\label{def:G}
    G(u,v) = \frac{L_f-\mu_f}{4}G_0(u,v)+\frac{\mu_f}{2}\left(\|u\|^2+\|v\|^2\right),
\end{equation}
where
\[
    G_0(u, v)\defeq\left(\alpha-u_1\right)^2+\left(v_1-u_2\right)^2+\left(v_2-u_3\right)^2+\cdots+\left(v_{d-1}-u_d\right)^2.
\]
The construction is based on Nesterov's well-known “chain-like” quadratic function \cite{nesterov2018lectures}. In the following, we give some properties of the constructed problem instance.
\begin{lemma}\label{lemma:SCproperty}
The problem defined in \eqref{eqn:OBJ} has the following properties.
    \begin{enumerate}
        \item \label{Lmu} $f_0(x)$ is $L_f$-smooth and $\mu_f$-strongly convex.
        \item \label{optx} Denote $q=\frac{\sqrt{\kappa_f}-1}{\sqrt{\kappa_f}+1}$.
        Then the optimal solution $x^\star$ of \eqref{eqn:OBJ} is given by
        \begin{equation}\label{v*}
            x^\star[1]=x^\star[2]=\cdots=x^\star[2N]= \mathbbm{x}^\star,
        \end{equation}
        where $ \mathbbm{x}^\star\in\RR^{d}$ is given by
        \begin{equation}\label{def:vstar}
            \mathbbm{x}^\star_i=\alpha\cdot \frac{q^i+q^{2d+1-i}}{1+q^{2d+1}}, \qquad i=1,\cdots,d.
        \end{equation}
        Moreover, for any $K\geq0$, we have $\sum_{i=K+1}^d ( \mathbbm x^\star_i)^2 \geq q^{2K} \cdot \frac{d-K}{d}\|\mathbbm{x}^\star\|^2$.
    \end{enumerate}
\end{lemma}
\begin{proof}
    Part \ref{Lmu}:
    Let us fix the vectors $u,v,\omega,\nu\in\RR^d$ with $\|(\omega,\nu)\|=1$, and define $h(\theta)\defeq G_0(u+\theta \omega,v+\theta \nu)$. 
    If we take $v_0=0$ and $\nu_0=0$, it holds
        \begin{align*}
            h^{\prime\prime}(0)=&\sum_{i=1}^d \frac{\partial^2G_0(u,v)}{\partial u_i^2}\omega_i^2+\sum_{i=0}^{d-1} \frac{\partial^2G_0(u,v)}{\partial v_i^2}\nu_i^2+2\sum_{i=1}^d \frac{\partial^2G_0(u,v)}{\partial u_i\partial v_{i-1}}\omega_i \nu_{i-1}\\
            =& 2\sum_{i=1}^d \omega_i^2+2\sum_{i=0}^{d-1}\nu_i^2-4\sum_{i=1}^d \omega_i \nu_{i-1}.
        \end{align*}
    On the one hand, due to Cauchy inequality, we have $h^{\prime\prime}(0)\geq 0$.
    On the other hand, 
    \[
        h^{\prime\prime}(0)\leq 2\sum_{i=1}^d \omega_i^2+2\sum_{i=0}^{d}\nu_i^2+2\sum_{i=1}^d (\omega_i^2+ \nu_{i-1}^2)\leq4,
    \]
    where the last inequality follows from $\|(\omega,\nu)\|=1$.
    Therefore, $G_0(u,v)$ is convex and $4$-smooth, 
    hence $G(u,v)$ is $\mu_f$-strongly convex and $L_f$-smooth, 
    which implies $f_0(x)$ is also $\mu_f$-strongly convex and $L_f$-smooth.

    Part \ref{optx}: 
    For any $x$ satisfies the constraint in \eqref{eqn:OBJ}, we have $x=(v,v,\cdots,v)$ and $f_0(x)=NG(v,v)$.
    Thus, we only need to verify that $\mathbbm{x}^\star$ is the (unique) minimum point of the function $G(v,v)$.
    By the optimality condition $\nabla_v G(v,v)=0$, we have
    \[
    \left\{\begin{array}{l}
        (2+\beta) v_1^\star-v_2^\star=\alpha, \\
        -v_1^\star+(2+\beta) v_2^\star-v_3^\star=0, \\
        \qquad \qquad \vdots \\
        -v_{d-2}^\star+(2+\beta) v_{d-1}^\star-v_n^\star=0, \\
        -v_{d-1}^\star+(1+\beta) v_d^\star=0,
    \end{array}\right.
    \]
    where we denote $\beta\defeq \frac{4\mu_f}{L_f-\mu_f}$. Note that $q=\frac{1}{2}\left((2+\beta)-\sqrt{(2+\beta)^2-4}\right)$ is the smallest root of the quadratic equation $\lambda^2-(2+\beta)\lambda+1=0$.
    By a direct calculation, we can check that $\mathbbm{x}^\star$ satisfies the $d$ equations, and hence $\mathbbm{x}^\star$ is the minimum of the function $G(v,v)$.
    Lastly, it holds
    \begin{align*}
        \sum_{i=K+1}^d (\mathbbm x_i^\star)^2 
        = &\alpha^2\sum_{i=K+1}^d\frac{q^{2i}+q^{4d+2-2i}+2q^{2d+1}}{(1+q^{2d+1})^2}\\
        \stackrel{(i)}{=} & \frac{\alpha^2}{(1+q^{2d+1})^2}\left(2q^{2d+1}(d-K)+\sum_{i=K+1}^{2d-K}q^{2i}\right)\\
        \stackrel{(ii)}{\geq} & \frac{\alpha^2}{(1+q^{2d+1})^2}\left(2q^{2d+1}(d-K)+\frac{2d-2K}{2d}\sum_{i=1}^{2d}q^{2i+2K}\right)\\
        \geq & \frac{\alpha^2(d-K)q^{2K}}{d(1+q^{2d+1})^2}  \left(2d q^{2d+1}+\sum_{i=1}^{2d}q^{2i}\right)\\
        = & q^{2K} \cdot \frac{d-K}{d}\|\mathbbm{x}^\star\|^2,
    \end{align*}
    where $(i)$ follows from $\sum_{i=K+1}^d (q^{2i}+q^{4d+2-2i})=\sum_{i=K+1}^{d}q^{2i}+\sum_{i=d+1}^{2d-K}q^{2i}=\sum_{i=K+1}^{2d-K}q^{2i}$,
    $(ii)$ holds because $q<1$ and $\frac{1}{2d-2K}\sum_{i=K+1}^{2d-K}q^{2i}\geq \frac{1}{2d}\sum_{i=K+1}^{2d+K}q^{2i}=\frac{1}{2d}\sum_{i=1}^{2d}q^{2i+2K}$.\qed
\end{proof}
Now we are ready to give our lower bound result for the strongly convex case.

\begin{theorem}\label{lbd:sc}
    Let parameters $L_f>\mu_f>0,\kappa_A\geq 1$ be given. For any integer $k\geq \left\lfloor\frac{\kappa_A+1}{2}\right\rfloor$, there exists an instance in $\mathcal{F}_{\mathrm{SC}}(L_f,\mu_f,\kappa_A,D)$ of form \eqref{eqn:OBJ}, with component function $G$ defined in \eqref{def:G}, $N = \left\lfloor\frac{\kappa_A+1}{2}\right\rfloor, K=\left\lfloor \frac{k-2}{N+1} \right\rfloor+1, d = 2K$ and $\alpha$ be suitably chosen. For this problem, the iterates generated by any first-order algorithm in $\mathcal{A}$ satisfies
    \[
        \suboptsc(x_k)\geq \frac12\paren{\frac{\sqrt{\kappa_f}-1}{\sqrt{\kappa_f}+1}}^{\frac{2k}{N}}\|x_0-x^\star\|.
    \]
\end{theorem}

\begin{proof}
    By property \eqref{Lmu}, we know $f_0(x)$ is $L_f$-smooth and $\mu_f$-strongly convex.
    The condition number of $A$ is not great than $\sqrt{2N^2-1}\leq\sqrt{\frac{(\kappa_A+1)^2}{2}-1}\leq \kappa_A$, and $\alpha$ can be suitably chosen so that $\|x_0-\xs\|=D$.
    Thus the instance we construct belongs to the problem class $\mathcal{F}_{\mathrm{SC}}(L_f,\mu_f,\kappa_A)$.
    According to and Lemma \ref{lemma:index} and our choice of $K$, we have $(x_k[j])_s=0$ for any $N+1\leq j\leq 2N$ and $ K+1\leq s\leq d$.
    Therefore,
    \[
        \left\|x_k-x^{\star}\right\|^2 \geq N \times \sum_{s=K+1}^d\left(\mathbbm x_s^{\star}\right)^2 \geq N \times q^{2K} \cdot \frac{d-K}{d}\left\|\mathbbm x^{\star}\right\|^2,
    \]
    where the last inequality comes from Property \ref{optx}.
    Notice that $\left\|x_0-x^{\star}\right\|^2=2 N\left\|\mathbbm x^\star\right\|^2$ and $K=\frac{d}{2}$.
    Substituting them into the above inequality yields 
    \begin{equation}\label{eqn:q2k}
        \left\|x_k-x^{\star}\right\|^2 \geq \frac{q^{2K}}{4}\left\|x_0-x^{\star}\right\|^2.
    \end{equation}
    Plugging in the definition $q=\frac{\sqrt{\kappa_f}-1}{\sqrt{\kappa_f}+1}$ and the fact that $K\leq \frac{2k}{N}$ completes the proof. \qed
\end{proof}
\begin{corollary}
    For any first-order algorithm in $\mathcal{A}$, parameters $\kappa_f\geq 2$ and $0<\epsilon\leq\frac{D}{20}$, there exists a problem instance in $\mathcal{F}_{\mathrm{SC}}(L_f,\mu_f,\kappa_A,D)$ such that at least
    \[
        \Omega\left(\kappa_A\sqrt{\kappa_f} \log\left(\frac{D}\epsilon\right)\right)
    \]
    iterations are required in order to find an iterate $x_k$ satisfies $\suboptsc(x_k)\leq\epsilon$.
\end{corollary}

\subsection{Non-convex case}

For the hard instance we present in \eqref{eqn:OBJ}, the definition of suboptimal measure becomes
\[
    \suboptnc(x)\defeq L_f\left\| x-\mathcal{P}_{\mathcal{X}}\left(x-\frac{1}{2L_f}\nabla f_0(x)\right)\right\|,
\]
where $\mathcal{X}$ refers to the feasible region of problem \eqref{eqn:OBJ}, i.e., $\mathcal{X}\defeq\{x\in\RR^{2Nd}\mid x[1]=x[2]=\cdots=x[2N]\}$.
For given positive integer $d$, we define the function $G_0(\cdot,\cdot): \RR^d\times\RR^d\mapsto \RR$ as
\begin{equation}\label{def:ncg0}
    G_0(u,v)=-\Psi(1)\Phi(u_1)+\sum_{i=2}^{d} \left[\Psi(-v_{i-1})\Phi(-u_i)-\Psi(v_{i-1})\Phi(u_i)\right],
\end{equation}
where function $\Phi(x)$ and $\Psi(x)$ are defined as
\[
    \Psi(x)\defeq \begin{cases}0 & x \leq 1 / 2 \\ \exp \left(1-\frac{1}{(2 x-1)^2}\right) & x>1 / 2\end{cases}
    \quad \text{and} \quad
    \Phi(x)=\sqrt{e} \int_{-\infty}^x e^{-\frac{1}{2} t^2} \mathrm d t.
\]
 The function $G(\cdot,\cdot)$ in formulation \eqref{eqn:OBJ} is a scaled version of $G_0(\cdot,\cdot)$ and its formal definition will be given later. Let us discuss $G_0(\cdot,\cdot)$ first. Define $g_0(v)\defeq G_0(v,v)$ and one can observe $g_0(v)$ coincides with the hard instance constructed in \cite{carmon2020lower}.
We give some useful properties of $G_0(u,v)$.
\begin{lemma}\label{prop-G0}
The function $G_0(u,v)$ has the following properties.
    \begin{enumerate}
        \item If $v_d=0$, then $\|\nabla g_0(v)\|\geq 1$. \label{largegrad}
        \item $G_0(0,0)-\inf_{u,v} G_0(u,v)\leq 12d$. \label{Delta}
        \item $G_0(u,v)$ is $l_0$-smooth with $l_0$ being a universal constant that is independent of the problem dimension and other constants.\label{smooth}
    \end{enumerate}
\end{lemma}
\begin{proof}
    Part \ref{largegrad} comes from \cite{carmon2020lower}.
    
    Part \ref{Delta}: Note that $0 \leq \Psi(x) \leq e$ and $0 \leq \Phi(x) \leq \sqrt{2 \pi e}$. it holds
    \[
        G_0(u,v)\geq-\Psi(1)\Phi(u_1)-\sum_{i=2}^{d} \Psi(v_{i-1})\Phi(u_i)\geq -d e\sqrt{2\pi e}\geq -12 d.
    \]
    Combining the fact that $G_0(0,0)\leq 0$ yields the property.

    To prove Part \ref{smooth}, fix $u,v,p,q\in\RR^d$ with $\|(p,q)\|=1$, define the function $h(\cdot):\RR\mapsto\RR$ by the directional projection of $G_0(u,v)$ along $(p,q)$, i.e., $h(\theta)\defeq G_0(u+\theta p,v+\theta q)$. Taking $v_0=1$ and $q_0=0$, We have
    \[
        h^{\prime\prime}(0)=\sum_{i=1}^d \frac{\partial^2G_0(u,v)}{\partial u_i^2}p_i^2+\sum_{i=1}^{d-1} \frac{\partial^2G_0(u,v)}{\partial v_i^2}q_i^2+2\sum_{i=2}^d \frac{\partial^2G_0(u,v)}{\partial u_i\partial v_{i-1}}p_i q_{i-1}.
    \]
    By simple derivations, we can obtain $0<\Psi(x)<e, 0 \leq \Psi^{\prime}(x) \leq \sqrt{54 / e}, |\Psi^{\prime\prime}(x)|\leq 8^5$ and $0<\Phi(x)<\sqrt{2\pi e}, 0<\Phi^{\prime}(x)\leq \sqrt{e}, \sup_x |\Phi^{\prime\prime}(x)|\leq 27$. It follows
    \begin{align*}
        \left|\frac{\partial^2G_0(u,v)}{\partial u_i^2}\right|&=\left|\Psi(-v_{i-1})\Phi^{\prime\prime}(-u_i)-\Psi(v_{i-1})\Phi^{\prime\prime}(u_i)\right|\leq 54e,\quad 1\leq i\leq d,\\
        \left|\frac{\partial^2G_0(u,v)}{\partial v_i^2}\right|&=\left|\Psi^{\prime\prime}(-v_{i})\Phi(-u_{i+1})-\Psi^{\prime\prime}(v_{i})\Phi(u_{i+1})\right|\leq 2\sqrt{2\pi e}\times 8^5,\quad 1\leq i\leq d-1,\\
        \left|\frac{\partial^2G_0(u,v)}{\partial u_i\partial v_{i-1}}\right|&=\left|\Psi^{\prime}(-v_{i-1})\Phi^\prime(-u_i)-\Psi^{\prime}(v_{i-1})\Phi^\prime(u_i)\right|\leq 2\sqrt{54},\quad 2\leq i\leq d.
    \end{align*}
    Therefore, it holds
    \begin{align*}
        |h^{\prime\prime}(0)| & \leq \max\left\{\left|\frac{\partial^2G_0(u,v)}{\partial u_i^2}\right|,\left|\frac{\partial^2G_0(u,v)}{\partial v_i^2}\right|,2\left|\frac{\partial^2G_0(u,v)}{\partial u_i\partial v_{i-1}}\right|\right\}\sum_{i=1}^d \left(p_i^2+q_i^2+p_i q_{i-1}\right)\\
        & \leq 2\sqrt{2\pi e}\times 8^5 \cdot \sum_{i=1}^d \left(p_i^2+q_i^2+p_i q_{i-1}\right).
    \end{align*}
    Since $\sum_{i=1}^d(p_i^2+q_i^2)=1$ and $\sum_{i=1}^d p_i q_{i-1}\leq \|p\|\|q\|\leq1$, we obtain
    \[
        |h^{\prime\prime}(0)|\leq 600000,
    \]
    which completes the proof. \qed
\end{proof}

 For given positive parameters $c_0, L_f,\alpha>0$, we construct the following scaled function based on $G_0(u,v)$,
\begin{equation}\label{def:ncg}
G(u,v)=\frac{L_f\alpha^2}{l_0}G_0\left(\frac{u}{\alpha},\frac{v}{\alpha}\right),
\end{equation}
where $\alpha$ can be adjusted to fulfill the condition $f_0(0)-\inf_x f_0(x)\leq\Delta$.  
Similarly, we define $g(v)\defeq G(v,v)$.
By simple derivations, we can generalize Lemma \ref{prop-G0} to the case of $G(u,v)$.
\begin{corollary}\label{prop-G}
The function $G(u,v)$ has the following properties.
    \begin{enumerate}
        \item $G(0,0)-\inf_{u,v} G(u,v)\leq 12dl_0^{-1}L_f\alpha^2$. \label{eqn:delta}
        \item If $v_d=0$, then $\|\nabla g(v)\|\geq l_0^{-1}L_f\alpha$. \label{eqn:grad}
        \item $G(u,v)$ is $L_f$-smooth. \label{eqn:smooth}
    \end{enumerate}
\end{corollary}

Next, we given a key lemma that gives the relationship between $\suboptnc$ and $\|\nabla g\|$.
\begin{lemma}\label{lemma:subopt}
    For any $x\in\RR^{2Nd}$, let $\bar x=\frac 1{2N}\sum_{i=1}^{2N}x[i]$. Suppose that $G(\cdot,\cdot)$ is $L_f$-smooth, then we have
    \[
        \suboptnc\geq \frac{\sqrt{N}}{4}\|\nabla g(\bar x)\|.
    \]
\end{lemma}
\begin{proof}
    Denote $\bar{\mathcal G}(x)\defeq \frac1{2N}\sum_{i=1}^{2N}\nabla_{x[i]}f_0(x)$.
    On the one hand, it holds that
    \[
        \mathcal{P}_{\mathcal{X}}\left(x-\frac{1}{2L_f}\nabla f_0(x)\right) = \left(\bar x-\frac1{2L_f} \bar{\mathcal G}(x), \bar x-\frac1{2L_f} \bar{\mathcal G}(x),\cdots,\bar x-\frac1{2L_f} \bar{\mathcal G}(x)\right).
    \]
    Therefore, we have
    \begin{equation}\label{eqn:NC-subopt}
    \begin{split}
        \left\| x-\mathcal{P}_{\mathcal{X}}\left(x-\frac{1}{2L_f}\nabla f_0(x)\right)\right\|^2
        &=\sum_{i=1}^{2N}\left\|x[i]-\bar x +\frac1{2L_f}\bar{\mathcal G}(x)\right\|^2\\
        &=\sum_{i=1}^{2N}\|x[i]-\bar x\|^2+\frac{N}{2L_f^2}\|\bar{\mathcal G}(x)\|^2,
    \end{split}
    \end{equation}
    where the last equality holds because $\sum_{i=1}^{2N} (x[i]-\bar x)=0$.
    On the other hand, recall that $g(x)=G(x,x)$, by chain rule we have
    \begin{equation}\label{nablag}
        \begin{split}
            N\|\nabla g(\bar x)\| = & N\left\|\nabla_{u}G(\bar x,\bar x)+\nabla_v G(\bar x,\bar x)\right\|\\
            \leq & \sum_{i=1}^N \left\| \nabla_u G(\bar x, \bar x)-\nabla_u G(x[i],x[i+N])\right\|\\
            &+\sum_{i=1}^N\left\| \nabla_v G(\bar x, \bar x)-\nabla_v G(x[i],x[i+N])\right\|\\
            &+\left\|\sum_{i=1}^N\left(\nabla_u G(x[i],x[i+N])+\nabla_v G(x[i],x[i+N])\right)\right\|.
        \end{split}
    \end{equation}
    For the first term on the right hand side, we have
    \begin{align*}
        &\left\| \nabla_u G(\bar x, \bar x)-\nabla_u G(x[i],x[i+N])\right\|\\
        \leq&~ \left\| \nabla_u G(\bar x, \bar x)-\nabla_u G(x[i],\bar x)\right\|+\left\| \nabla_u G(x[i], \bar x)-\nabla_u G(x[i],x[i+N])\right\|\\
        \leq&~ L_f\left(\|\bar x-x[i]\|+\|\bar x-x[i+N]\|\right),
    \end{align*}
    where the second inequality is due to the $L_f$-smoothness of $G(\cdot,\cdot)$.
    Similar derivation also applies to the second term on the right hand side of \eqref{nablag}. 
    Therefore, it holds
    \begin{align*}
        N\|\nabla g(\bar x)\|&\leq 2L_f\sum_{i=1}^{2N}\|\bar x -x[i]\|+2N\|\bar{\mathcal G}(x)\|\\
        &\stackrel{(i)}\leq 2L_f\sqrt{2N\sum_{i=1}^{2N}\|\bar x-x[i]\|^2}+2N\|\bar{\mathcal G}(x)\|\\
        &\stackrel{(ii)}\leq \sqrt{8L_f^2N+8L_f^2N}\sqrt{\sum_{i=1}^{2N}\|x[i]-\bar x\|^2+\frac{N}{2L_f^2}\|\bar{\mathcal G}(x)\|^2}\\
        &\stackrel{(iii)}\leq 4\sqrt N \cdot\suboptnc(x),
    \end{align*}
    where $(i)$ and $(ii)$ come from Cauchy–Schwarz inequality, $(iii)$ utilizes equality \eqref{eqn:NC-subopt}. \qed
\end{proof}

We can now state a complexity lower bound for finding an approximate solution of non-convex problems with first-order linear span algorithms.
\begin{theorem}\label{thm:NCLB}
    Let parameters $L_f,\Delta>0, \kappa_A\geq 1$ be given. For any integer $k\geq \kappa_A$, there exists an instance in $\mathcal{F}_\mathrm{NC}(L_f,\kappa_A,\Delta)$ of form \eqref{eqn:OBJ},
    with component function $G$ defined in \eqref{def:ncg0}, $N = \left\lfloor\frac{\kappa_A+1}{2}\right\rfloor,  K=\left\lfloor \frac{k-2}{N+1} \right\rfloor+1, d=K+2$ and $\alpha=\sqrt{\frac{l_0\Delta}{12NdL_f}}$.
    For this problem, suppose the initial point $x_0=0$, then the iterates generated by any first-order algorithm in $\mathcal{A}$ satisfies
    \[
        \suboptnc(x_k)\geq  c_0\cdot\sqrt{\frac{\kappa_AL_f\Delta}{k}}.
    \]
\end{theorem}
\begin{proof}
    By property \eqref{eqn:smooth} in Corollary \ref{prop-G}, we know $f_0(x)$ is $L_f$-smooth.
    According to the definition of $\alpha$, it holds $F(0)-\inf_x F(x)\leq f_0(0)-\inf_x f_0(x)\leq 12Ndl_0^{-1}L_f\alpha^2\leq \Delta$.
    As proved in Theorem \ref{lbd:sc}, the condition number of $A$ is not great than $\kappa_A$.
    Accordingly, the instance problem belongs to the class $\mathcal{F}_\mathrm{NC}(L_f,\kappa_A,\Delta)$.
    Since $K= d-2$, the last element of all $2N$ vectors is zero, that is, $(x[i])_d=0, \forall 1\leq i\leq 2N$.
    It implies $\bar x_d=0$. Utilizing Lemma \ref{lemma:subopt} and property \eqref{eqn:grad} in Corollary \ref{prop-G}, we obtain
    \[
        \suboptnc(x)\geq \frac{\sqrt{N}}{4}\|\nabla g(\bar x)\|\geq \frac{\sqrt{N}}{4}l_0^{-1}L_f\alpha= \frac{1}{4}\sqrt{\frac{L_f\Delta}{12l_0d}}.
    \]
    By definition, we have $d=K+2\leq\frac{5k}{\kappa_A}$, and this gives the desired result. \qed
\end{proof}

\begin{corollary}
    For any first-order algorithm in $\mathcal{A}$ and $0<\epsilon\leq c_0\sqrt{L_f\Delta}$ where $c_0$ is defined in Theorem \ref{thm:NCLB}, there exists a problem instance in $\mathcal{F}_\mathrm{NC}(L_f,\kappa_A,\Delta)$ such that at least 
    \[
        \Omega\left(\frac{\kappa_A L_f\Delta}{\epsilon^2}\right)
    \]
    iterations are required in order to find an iterate $x_k$ satisfies $\suboptnc(x_k)\leq\epsilon$.
\end{corollary}

\subsection{Convex case}
For the linear equality constrained problem, the definition of suboptimality becomes
\begin{align*}
    \suboptc(x)=f(x)-\min_{x:Ax=b}f(x)+\rho\|Ax-b\|.
\end{align*}
Recall that for $\rho\geq 2\nrm{y^\star}$, $\suboptc(x)$ implies the standard optimality measure $\max\{\abs{f(x)-f(x^\star)},\nrm{Ax-b}\}$.

Note that in Section \ref{sec:SCLB}, we consider the strongly convex problem class $\mathcal{F}(L_f,\mu_f,\kappa_A)$ with $\mu_f>0$. In this section, we demonstrate how the complexity lower bound present in Theorem \ref{lbd:sc} can be reduced to the convex problem class with $\mu_f=0$.
\begin{theorem}\label{thm:CLB}
    Let parameters $L_f,D>0, \kappa_A\geq 1$ be given. For any integer $k\geq \kappa_A$, there exists an instance in $\mathcal{F}_\mathrm{C}(L_f,\kappa_A,D)$ of form \eqref{eqn:OBJ},
    with component function $G$ defined in \eqref{def:G}, $N = \left\lfloor\frac{\kappa_A+1}{2}\right\rfloor,  K=\left\lfloor \frac{k-2}{N+1} \right\rfloor+1, d=2K, \mu_f = \frac{L_f}{(K+1)^2}$ and $\alpha$ be suitably chosen.
    For this problem, the $k$-th iterate $x_k$ generated by any first-order algorithm in $\mathcal{A}$ satisfies
    \[
        f(x_k)-f(x^\star)-\iprod{\lambda^\star}{Ax_k-b}\geq c_0\cdot \frac{\kappa_A^2 L_f D^2}{k^2},
    \]
    where $\lambda^\star$ is the optimal dual variable and $c_0$ is a universal constant.
\end{theorem}
\begin{proof}
    We can suitably choose $\alpha$ so that $\|x_0-\xs\|=D$. Then, under our choice of parameters, we have $\kappa_f=L_f/\mu_f=(K+1)^2$. Note that $q\geq \exp\left(-\frac{2}{\sqrt{\kappa_f}-1}\right)$, and hence $q^{2K}\geq e^{-4}$. Therefore, using the fact that $f$ is $\mu_f$-strongly convex and $A^\T \lams = \nabla f(\xs)$, we have
    \begin{align*}
        f(x_k)-f(x^\star)-\iprod{\lambda^\star}{Ax_k-b}
        =&~
        f(x_k)-f(x^\star)-\iprod{\nabla f(\xs)}{x_k-\xs} \\
        \geq&~
        \frac{\mu_f}{2}\left\|x_k-x^{\star}\right\|^2 
        \geq
        \frac{\mu_f q^{2K}}{8}\left\|x_0-x^{\star}\right\|^2 \\
        \geq&~
        \frac{e^{-4}}{8}\cdot \frac{L_f D^2}{(K+1)^2},
    \end{align*}
    where the last two inequalities are due to \eqref{eqn:q2k}.
    According to the choice of $K$ and $N$, we have $K+1\leq \frac{4k}{\kappa_A}$, and hence complete proof. \qed
\end{proof}

\begin{corollary}
    For any first-order algorithm in $\mathcal{A}$ and any $0<\epsilon\leq c_0L_fD^2$ where $c_0$ is defined in Theorem \ref{thm:CLB}, there exists a problem instance in $\mathcal{F}_\mathrm{C}(L_f,\kappa_A,D)$ such that at least 
    \[
        \Omega\left(\frac{\kappa_A \sqrt{L_f}D}{\sqrt{\epsilon}}\right)
    \]
    iterations are required in order to find an iterate $x_k$ satisfies $\suboptc(x_k)\leq\epsilon$ with $\rho\geq\nrm{\lams}$.
\end{corollary}

\section{Linear equality constrained problem}\label{sec:linear-equality}
In this section, we focus on a special case of \eqref{comp-prob} when $h(\cdot)=\mathbb{I}_{\{0\}}(\cdot)$, corresponding to the linear equality constrained problem
\begin{equation}\label{equality-prob}
\min_x\ f(x),\quad\st\ Ax=b.
\end{equation}
In Section \ref{sec:sigmin}, it will be demonstrated that the requirement of full row rank for the matrix $A$ can be removed. 
In Section \ref{sec:direct-acceleration}, we will show that the optimal iteration complexity can be achieved by directly applying APPA on the convex problem.

\subsection{Refined result of linear equality constrained problem}\label{sec:sigmin}
First, we consider the strongly convex case. Similar with the derivation in Section \ref{sec:SC-ub}, problem \eqref{equality-prob} is equivalent to
\[
    \min_x\max_\lambda\ \mathcal{L}(x,\lambda)\defeq f(x)+\lambda^\T (Ax-b).
\]
Due to the strong duality, we have $\min_x\max_\lambda\ \mathcal{L}(x,\lambda)=\max_\lambda\min_x\ \mathcal{L}(x,\lambda)$. 
Accordingly, the corresponding dual problem can be written as 
\begin{equation}\label{eqn:dual}
    \max_\lambda\ \Phi(\lambda)\defeq -b^\T \lambda-f^\star(-A^\T \lambda).
\end{equation}
Without the assumption that $A$ is full row rank, the proof of Lemma \ref{lemma:Phi-SC} no longer holds and $\Phi(\lambda)$ is not necessarily strongly convex. Actually, we have the following property.
\begin{lemma}\label{lemma:Phi-SC2}
     Let $\sigmin$ be the nonzero minimal singular value of $A$. It holds that $\Phi(\lambda)$ is $(\sigmin^2/L_f)$-strongly concave in the subspace $\mathcal{R}(A)$.
\end{lemma}
\begin{proof}
Noting that $f(x)$ is $L_f$-smooth and $\mu_f$-strongly convex,
it is easy to derive that $f^\star(x)$ is $\frac 1{\mu_f}$-smooth and $\frac 1{L_f}$-strongly convex.
Then for any $\lambda,\lambda^\prime\in\mathcal{R}(A)$, we have 
\begin{align*} 
    \Phi(\lambda^\prime)-\Phi(\lambda) & = -\langle b,\lambda^\prime-\lambda\rangle-f^\star(-A^\T \lambda^\prime)+f^\star(-A^\T \lambda)\\
    & \leq -\langle b, \lambda^\prime-\lambda\rangle+\langle \nabla f^\star(-A^\T \lambda), A^\T \lambda^\prime-A^\T \lambda\rangle -\frac {\|A^\T \lambda-A^\T \lambda^\prime\|^2}{2L_f} \\
    & = \langle \nabla\Phi(\lambda), \lambda^\prime-\lambda\rangle-\frac 1{2L_f} \|A^\T \lambda-A^\T \lambda^\prime\|^2\\
    & \leq \langle \nabla\Phi(\lambda), \lambda^\prime-\lambda\rangle-\frac{\sigmin^2}{2L_f}\|\lambda-\lambda^\prime\|^2,
\end{align*}
where the last inequality holds because $\lambda,\lambda^\prime\in\mathcal{R}(A)$. \qed
\end{proof}
According to the update rule of Algorithm \ref{algo:SC-APPA},
if we set $\lambda_0=0$, then we have
\[
    \lambda_{k+1}\in \mathrm{Span}\left\{Ax_0-b, Ax_1-b,\cdots, Ax_k-b\right\}, k\geq 0.
\]
Due to the feasibility of the constraint $Ax=b$, we have $b\in\mathcal{R}(A)$. Hence it holds that the iterate sequence $\{\lambda_k\}$ always stays in $\mathcal{R}(A)$. Combining Lemma \ref{lemma:Phi-SC2}, we can treat $\Phi(\lambda)$ as a strongly concave function during the iterations. Then the derivation in Section \ref{sec:SC-ub} still holds by replacing $\sigmamin$ with $\sigmin$ and replacing $\kappa_A$ with $\underline\kappa_A\defeq\frac{\LA}{\sigmin}$. Consequently, the derived complexity upper bound is $\tilde O(\underline\kappa_A\sqrt{\kappa_f}\log(1/\epsilon))$.

For the convex problem, Algorithm \ref{algo:SC-APPA} is utilized to solve the strongly convex auxiliary problem \eqref{eqn:C-obj}, hence the upper bound can also be generalized to $\tilde O\left(\frac{\underline\kappa_A\sqrt{L_f}D}{\sqrt{\epsilon}}\right)$.
For the nonconvex problem, Algorithm \ref{algo:SC-APPA} is utilized to solve the subproblem \eqref{NC-PPA}. Therefore, the complexity upper bound can be similarly generalized to $\tilde O\left(\frac{\underline\kappa_A L_f \Delta}{\epsilon^2}\right)$.

\subsection{Direct acceleration for convex problem}\label{sec:direct-acceleration}
Recall that in Section \ref{sec:C-ub}, we derive the upper bound of the convex problem by constructing a strongly convex auxiliary problem. In this section, we propose an optimal algorithm, present in Algorithm \ref{APPA-convex}, that solves the original convex problem directly. The algorithm performs an inexact accelerated proximal point method on $f(x)$ while keeping the constraint intact in the outer loop. In the inner loop, it uses Algorithm \ref{algo:SC-APPA} to iteratively solve the subproblem until it meets the first-order suboptimality conditions.

\begin{algorithm2e}[H]\label{APPA-convex}
	\caption{Inexact APPA for convex problems}
    \textbf{Input:} initial point $x_0$, smoothness parameter $L_f$, condition number $\kappa_A$, subproblem error tolerances $\{\epsilon_k\}$ and $\{\gamma_k\}$, the maximum iteration number $T$.\\
	\textbf{Initialize:} $y_1\leftarrow x_0$, $t_1\leftarrow 1$.\\
    \For{$k=1\cdots T$}
    {
        Apply Algorithm \ref{algo:SC-APPA} to find a pair of suboptimal primal and dual variable $(x_k,\lambda_k)$ of problem
        \begin{equation}\label{proximal-point}
            \min_x f(x)+\frac {L_f}2\|x-y_k\|^2,\quad \st Ax=b,
        \end{equation}
        such that $\|\nabla f(x_k)+L_f(x_k-y_k)+A^\T\lambda_k\|\leq \sqrt{\frac {L_f}2}\cdot\frac{\epsilon_k}{t_k}$ and
        $\|Ax_k-b\|\leq \gamma_k$.\\
        Compute $t_{k+1}=\frac{1+\sqrt{1+4 t_k^2}}{2}$.\\
        Compute $y_{k+1}=x_k+\left(\frac{t_k-1}{t_{k+1}}\right)\left(x_k-x_{k-1}\right)$.
    }
    \textbf{Output:} $x_T$.
\end{algorithm2e}
The convergence proof is inspired by \cite{beck2009fast} and \cite{jiang2012inexact}.
Different from several work that studies the inexact accelerated PPA for the unconstrained convex problem, we have an additional linear equality constraint here. Since the projection onto the linear constraint is not allowed, we need to incorporate the violation of the constraint into the analysis, which makes the convergence proof quite different.
For simplicity, we define the following notations:
\[
    \nabla f(x_k)+L_f(x_k-y_k)+A^\T\lambda_k= \delta_k,~~ Ax_k-b= \zeta_k,
\]
\[
    v_k= f(x_k)-f(x^\star)+\langle\lambda^\star,Ax_k-b\rangle\geq 0,~~ u_k=t_k x_k-\left(t_k-1\right) x_{k-1}-x^\star,
\]
\[
    a_k=t_k^2 v_k \geq 0, ~~ b_k=\frac {L_f}2\|u_k\|^2, ~~ \tau = \frac {L_f}2\|x_0-x^\star\|^2, ~~ e_k=t_k\left\langle\delta_k, u_k\right\rangle,
\]
\[
    \omega_k=\|x_k-x^\star\|, ~~ \xi_k= |t_k\langle \lambda_k-\lambda^\star,Au_k\rangle|,~~ \bar{\epsilon}_k=\sum_{j=1}^k \epsilon_j,~~ \bar{\xi}_k=\sum_{j=1}^k\left(\xi_j+\epsilon_j^2\right).
\]

Due to the KKT condition, we know $Ax^\star=b$ and $A^\T \lambda^\star = \nabla f(x^\star)$. It follows that $v_k= f(x_k)-f(x^\star)-\langle \nabla f(x^\star), x-x^\star\rangle$, which corresponds to the Bregman divergence associated with $f$ for point $x_k$ and $x^\star$.
In other words, $v_k$ measures the distance between $x_k$ and $x^\star$ under the Bregman divergence and can somewhat serves as a suboptimality measure. In the following analysis, we will give an upper bound of $v_k$ (or $a_k$) first and eventually derive an upper bound for the objective function gap and the constraint violation.

According to the update rule $t_{k+1}=\frac{1+\sqrt{1+4 t_k^2}}{2}$ and $t_1=1$, it is easy to obtain the following lemma, which will be frequently used in the following proofs.
\begin{lemma}
    The sequence $\{t_k\}$ generated by Algorithm \ref{APPA-convex} satisfies
    $\frac{k+1}{2}\leq t_k\leq k $ for any $k\geq1$.
\end{lemma}

In the following lemma, we give the one-step estimation of $a_k+b_k$.
\begin{lemma}
For any $k\geq1$, it holds
\begin{equation}\label{potential}
    a_{k+1}+b_{k+1}\leq a_k+b_k+\xi_{k+1}+e_{k+1}.
\end{equation}
\end{lemma}
\begin{proof}
    For each $j$ and any $x$, we have 
    \begin{equation}\label{descent-bound}
        \begin{split}
        f(x)-f(x_j) &\geq \langle x-x_j, \nabla f(x_j)\rangle\\
        &= \langle x-x_j, \delta_j-L_f(x_j-y_j)-A^\T\lambda_j\rangle\\
        &\geq \frac {L_f}2\|x_j-y_j\|^2+L_f\langle x_j-y_j,y_j-x\rangle\\
        &\quad +\langle x-x_j,\delta_j-A^\T\lambda_j\rangle.
        \end{split}
    \end{equation}
Note that $v_k-v_{k+1}=f(x_k)-f(x_{k+1})+\langle\lambda^\star,Ax_k-Ax_{k+1}\rangle$.
We apply \eqref{descent-bound} with $j=k+1$ and $x = x_k$ to get
\begin{equation}\label{eq:xk}
    \begin{split}
    v_k-v_{k+1}&\geq \frac {L_f}2\|x_{k+1}-y_{k+1}\|^2+L_f\langle x_{k+1}-y_{k+1},y_{k+1}-x_k\rangle\\
    &\quad +\langle \zeta_k-\zeta_{k+1},\lambda^\star-\lambda_{k+1}\rangle+\langle x_k-x_{k+1},\delta_{k+1}\rangle,
    \end{split}
\end{equation}
where we utilize the fact that $Ax_k-Ax_{k+1}=\zeta_k-\zeta_{k+1}$.
Similarly, let $j=k+1$ and $x=x^\star$ in \eqref{descent-bound}, we obtain
\begin{equation}\label{eq:xstar}
    \begin{split}
    - v_{k+1} &= f(x^\star)-f(x_{k+1})-\langle\lambda^\star,Ax_{k+1}-b\rangle\\
    &\geq \frac {L_f}2\|x_{k+1}-y_{k+1}\|^2+L_f\langle x_{k+1}-y_{k+1},y_{k+1}-x^\star\rangle\\
    &\quad-\langle Ax^\star-Ax_{k+1},\lambda_{k+1}\rangle-\langle\lambda^\star,Ax_{k+1}-b\rangle+\langle x^\star-x_{k+1},\delta_{k+1}\rangle\\
    &= \frac {L_f}2\|x_{k+1}-y_{k+1}\|^2+L_f\langle x_{k+1}-y_{k+1},y_{k+1}-x^\star\rangle\\
    &\quad-\langle \zeta_{k+1},\lambda^\star-\lambda_{k+1}\rangle+\langle x^\star-x_{k+1},\delta_{k+1}\rangle,
\end{split}
\end{equation}
Multiply \eqref{eq:xk} by $(t_{k+1}-1)$ and add it to \eqref{eq:xstar}.
Then, multiply $t_{k+1}$ on the both side of the obtained inequality and notice the relation $t_k^2=t_{k+1}^2-t_{k+1}$, we have
\begin{align*} 
    t_k^2 v_k-t_{k+1}^2v_{k+1}&\geq \frac {L_f}2\|t_{k+1}(x_{k+1}-y_{k+1})\|^2\\
    &\quad+L_ft_{k+1}\langle x_{k+1}-y_{k+1}, t_{k+1}y_{k+1}-(t_{k+1}-1)x_k-x^\star\rangle\\
    &\quad+\langle \lambda^\star-\lambda_{k+1},t_k^2\zeta_k-t_{k+1}^2\zeta_{k+1}\rangle-t_{k+1}\langle \delta_{k+1},u_{k+1}\rangle.
\end{align*}
For the first two terms on the right hand side of the above inequality,
we use the usual Pythagoras relation $\|{b}-{a}\|^2+2\langle{b}-{a}, {a}-{c}\rangle=\|{b}-{c}\|^2-\|{a}-{c}\|^2$,
then substitute $t_{k+1} y_{k+1}=t_{k+1} x_k+\left(t_k-1\right)\left(x_k-x_{k-1}\right)$ into it.
After rearranging, we obtain
\[
    t_{k+1}^2v_{k+1}+\frac {L_f}2\|u_{k+1}\|^2\leq t_k^2 v_k+\frac {L_f}2\|u_k\|^2-\langle \lambda^\star-\lambda_{k+1},t_k^2\zeta_k-t_{k+1}^2\zeta_{k+1}\rangle+e_{k+1}.
\]
Combining the fact $t_{k+1}Au_{k+1}=t_{k+1}^2\zeta_{k+1}-t_{k}^2\zeta_{k}$ yields the conclusion. \qed
\end{proof}

Now, having the inequality \eqref{potential}, we can derive an upper bound of $a_k$.
\begin{lemma}\label{lemma:ak}
    For any $k\geq 1$, it holds
    \begin{equation} 
        \label{eqn:ak} 
        a_k \leq\left(\sqrt{\tau}+\bar{\epsilon}_k\right)^2+2 \bar{\xi}_k.
    \end{equation}
\end{lemma}

\begin{proof}
By applying \eqref{descent-bound} with $x=x^\star$ and $j=1$, we have
\begin{align*}
    f(x^\star)-f(x_1) &\geq \frac {L_f}2\|x_1-y_1\|^2+L_f\langle x_1-y_1,y_1-x^\star\rangle\\
    &\quad-\langle Ax^\star-Ax_1,\lambda_1\rangle+\langle x^\star-x_1,\delta_1\rangle\\
    &= \frac {L_f}2\|x_1-x^\star\|^2-\frac {L_f}2\|y_1-x^\star\|^2\\
    &\quad-\langle Ax^\star-Ax_1,\lambda_1\rangle+\langle x^\star-x_1,\delta_1\rangle,
\end{align*}
where we utilize the usual Pythagoras relation $\|{b}-{a}\|^2+2\langle{b}-{a}, {a}-{c}\rangle=\|{b}-{c}\|^2-\|{a}-{c}\|^2$.
Noting that $y_1=x_0$, $t_1=1$ and $u_1=x_1-x^\star$, we have
\[
    t_1^2v_1 +\frac {L_f}2\|u_1\|^2 \leq \frac {L_f}2\|x_0-x^\star\|^2+t_1\langle \lambda^\star-\lambda_1, Au_1\rangle+e_1.
\]
Since $e_k\leq (t_k\sqrt{2/L_f}\|\delta_k\|)(\sqrt{L_f/2}\|u_k\|)\leq\epsilon_k\sqrt{b_k}$, the above inequality implies
\begin{equation}\label{potential1}
    a_1+b_1\leq \tau +\xi_1 + \epsilon_1\sqrt{b_1}.
\end{equation}
Let $s_k=\sum_{i=1}^{k} \xi_k+\sum_{i=1}^k \epsilon_i\sqrt{b_i}$.
Utilizing \eqref{potential} repeatedly and combining \eqref{potential1}, we obtain
\begin{equation}\label{eqn:repeat}
    a_k+b_k \leq \tau+s_k.
\end{equation}
Since $a_k\geq 0$, we have $s_k=s_{k-1}+\epsilon_k \sqrt{b_k}+\xi_k \leq s_{k-1}+\epsilon_k \sqrt{\tau+s_k}+\xi_k$, 
which implies
\[
    \sqrt{\tau+s_k} \leq \frac{1}{2}\left(\epsilon_k+\sqrt{\epsilon_k^2+4\left(\tau+s_{k-1}+\xi_k\right)}\right).
\]
By some simple derivations, we get
\begin{equation}\label{eqn:sk}
    \begin{split}
        s_k & \leq s_{k-1}+\frac{1}{2} \epsilon_k^2+\xi_k+\frac{1}{2} \epsilon_k \sqrt{\epsilon_k^2+4\left(\tau+s_{k-1}+\xi_k\right)} \\
        & \leq s_{k-1}+\frac{1}{2} \epsilon_k^2+\xi_k+\frac{1}{2} \epsilon_k \left(\epsilon_k+2\sqrt{\tau}+2\sqrt{s_{k-1}+\xi_k}\right) \\
        & \leq s_{k-1}+\epsilon_k^2+\xi_k+\epsilon_k\left(\sqrt{\tau}+\sqrt{s_{k-1}+\xi_k}\right).
    \end{split}
\end{equation}
From the inequality \eqref{potential1}, we know that $\tau \geq b_1-\epsilon_1 \sqrt{b_1}-\xi_1$, 
which follows $\sqrt{b_1} \leq \frac{1}{2}\left(\epsilon_1+\sqrt{\epsilon_1^2+4\left(\tau+\xi_1\right)}\right) \leq \epsilon_1+\sqrt{\tau+\xi_1}$.
Accordingly, 
\begin{equation}\label{eqn:s1}
    s_1=\epsilon_1 \sqrt{b_1}+\xi_1 \leq \epsilon_1\left(\epsilon_1+\sqrt{\tau+\xi_1}\right)+\xi_1 \leq \epsilon_1^2+\xi_1+\epsilon_1\left(\sqrt{\tau}+\sqrt{\xi_1}\right).
\end{equation}
Applying \eqref{eqn:sk} repeatedly and combining \eqref{eqn:s1} yields
\begin{align*}
    s_k & \leq s_1+\sum_{j=2}^k \epsilon_j^2+\sum_{j=2}^k \xi_j+\sqrt{\tau} \sum_{j=2}^k \epsilon_j+\sum_{j=2}^k \epsilon_j \sqrt{s_{j-1}+\xi_j} \\
    & \leq \sum_{j=1}^k \epsilon_j^2+\sum_{j=1}^k \xi_j+\sqrt{\tau} \sum_{j=1}^k \epsilon_j+\sum_{j=1}^k \epsilon_j \sqrt{s_j} \\
    & \leq \bar{\xi}_k+\sqrt{\tau} \bar{\epsilon}_k+\sqrt{s_k} \bar{\epsilon}_k,
\end{align*}
where the second inequality holds because $s_{j-1}+\xi_j\leq s_j$ and $\xi_1\leq s_1$ by the definition of $s_k$. The above inequality implies
\begin{equation}\label{eqn:sqrtsk}
    \sqrt{s_k} \leq \frac{1}{2}\left(\bar{\epsilon}_k+\left(\bar{\epsilon}_k^2+4 \bar{\xi}_k+4 \bar{\epsilon}_k \sqrt{\tau}\right)^{1 / 2}\right),
\end{equation}
and hence $s_k \leq \bar{\epsilon}_k^2+2 \bar{\xi}_k+2 \bar{\epsilon}_k \sqrt{\tau}$.
Substituting it into \eqref{eqn:repeat}, we obtain
\begin{equation*}
    a_k \leq  \tau+\bar{\epsilon}_k^2+2 \bar{\xi}_k+2 \bar{\epsilon}_k \sqrt{\tau}\leq\left(\sqrt{\tau}+\bar{\epsilon}_k\right)^2+2 \bar{\xi}_k,
\end{equation*}
which completes the proof. \qed
\end{proof}

Note that the right hand side of \eqref{eqn:ak} depends on $\xi_k$, which is still unclear. We further estimate this term in the following lemma.

\begin{lemma}\label{lemma:xik}
    Assume that $\gamma_k\leq1$, $\|\delta_k\|\leq 1$ and $t_{k+1}^2\gamma_{k+1}\leq t_k^2\gamma_k$ for any $k\geq 1$, we have
    \begin{align}
        \label{eqn:xi1} \xi_1 &\leq \frac{1}{\sigmin^2}(L_f\LA\omega_1+L_f(\LA D+1)+\LA)t_1^2\gamma_1,\\
        \label{eqn:xik} \xi_{k} &\leq  \frac2{\sigmin^2}(L_f\LA \omega_{k}+4L_f +\LA ) t_{k-1}^2\gamma_{k-1}, \quad\forall k\geq 2.
    \end{align}
\end{lemma}
\begin{proof}
By the update rule $y_{k+1}=x_k+\left(\frac{t_k-1}{t_{k+1}}\right)\left(x_k-x_{k-1}\right)$, it holds
\begin{align*}
    Ax_{k+1}-Ay_{k+1}=\zeta_{k+1}-\left(1+\frac{t_k-1}{t_{k+1}}\right)\zeta_k+\frac{t_k-1}{t_{k+1}} \zeta_{k-1}.
\end{align*}
Thus, we have $\|Ax_{k+1}-Ay_{k+1}\|\leq 4$ due to the fact that $\gamma_k\leq1$ and $1\leq t_k\leq t_{k+1}$. 
Combining the KKT condition $\nabla f(x^\star)+A^\T\lambda^\star=0$ with the definition of $\delta_{k+1}$, we obtain
\[
    A^\T(\lambda^\star-\lambda_{k+1})=\nabla f(x_{k+1})-\nabla f(x^\star)+L_f(x_{k+1}-y_{k+1})-\delta_{k+1}.
\]
Accordingly, it follows
\begin{equation}\label{eqn:AATlambda}
    \begin{split}
    &\|AA^\T(\lambda^\star-\lambda_{k+1})\|\\
    =& ~\|A(\nabla f(x_{k+1})-\nabla f(x^\star))+L_f(Ax_{k+1}-Ay_{k+1})-A\delta_{k+1}\|\\
    \leq & ~{L_f\LA \|x_{k+1}-x^\star\|}+4L_f +\LA ,
    \end{split}
\end{equation}
where the inequality uses the smoothness of $f$ and $\|\delta_k\|\leq 1$.
Since $\lambda^\star-\lambda_{k+1}\in\mathcal{R}(A)$,
the above inequality actually implies
\[
    \|\lambda^\star-\lambda_{k+1}\|\leq\frac1{\sigmin^2}\left({L_f\LA \|x_{k+1}-x^\star\|}+4L_f +\LA \right).
\]
It follows that for any $k\geq1$, 
\begin{equation}\label{xiub}
    \begin{split}
        \xi_{k+1}&=|\langle\lambda^\star-\lambda_{k+1},t_{k+1}Au_{k+1}\rangle|\\
        &\leq  \frac1{\sigmin^2}(L_f\LA \|x_{k+1}-x^\star\|+4L_f +\LA ) \|t_k^2\zeta_k-t_{k+1}^2\zeta_{k+1}\|\\
        &\leq  \frac2{\sigmin^2}(L_f\LA \|x_{k+1}-x^\star\|+4L_f +\LA ) t_k^2\gamma_k,
    \end{split}
\end{equation}
where the last inequality uses the $t_{k+1}^2\gamma_{k+1}\leq t_k^2\gamma_k^2$.

Similar to the derivations in \eqref{eqn:AATlambda} and \eqref{xiub}, we can also obtain
\begin{align*}
    &\|AA^\T(\lambda^\star-\lambda_1)\|\\
    =&~\|A(\nabla f(x_1)-\nabla f(x^\star))+L_f(Ax_1-Ay_1)-A\delta_1\|\\
    =&~\|A(\nabla f(x_1)-\nabla f(x^\star))+L_f(Ax_1-b)-L_f(Ax_0-Ax^\star)-A\delta_1\|\\
    \leq&~ L_f\LA \omega_1+L_f\gamma_1+L_f\LA D +\LA \|\delta_1\|\\
    \leq&~ L_f\LA \omega_1+L_f+L_f\LA D +\LA .
\end{align*}
Note that $u_1=x_1-x^\star$, we have $Au_1=Ax_1-b$ and 
\[
    \xi_1 \leq t_1^2\gamma_1\|\lambda_1-\lambda^\star\|\leq \frac{1}{\sigmin^2}(L_f\LA \omega_1+L_f(\LA D+1)+\LA )t_1^2\gamma_1,
\]
which completes the proof. \qed
\end{proof}

According to Lemma \ref{lemma:ak} and Lemma \ref{lemma:xik}, it remains an estimate of $\omega_k$ to get the upper bound of $a_k$.
In the discussion that follows, we prove that $\omega_k$ can be bounded by a linear increasing sequence.
Utilizing this property, we give the requirements on $\epsilon_k$ and $\gamma_k$, and obtain the following result.

\begin{theorem}\label{thm:convex}
    Suppose that $f(x)$ is convex, $\|x_0-x^\star\|\leq D$, the matrix $A$ satisfies $\nrm{A}\leq\LA$ and the minimum nonzero singular value of $A$ is no smaller than $\sigmin$. In Algorithm \ref{APPA-convex}, we set the subproblem error tolerances as $\epsilon_k \leq \min\left\{\frac{\sqrt{L_f}D}{2\sqrt{2}k^2},\sqrt{\frac 2{L_f}}\right\}$ and $\gamma_k\leq \min\left\{\frac{\sigmin^2L_fD^2}{8(L_f\LA (\varpi+D)+4L_f +\LA )},1\right\}\frac{1}{t_k^2(k+1)^3}$ where $\varpi$ is a constant defined in \eqref{def:varpi}. Let $\{x_k\}$ be the iterate sequence generated by Algorithm \ref{APPA-convex}. It holds that
    \[
        v_k\leq \frac{16L_fD^2}{(k+1)^2}.
    \]
\end{theorem}

\begin{proof}
We can check that the assumptions in Lemma \ref{lemma:ak} are satisfied: $\gamma_k\leq1$, $\|\delta_k\|=\frac{1}{t_k}\sqrt{\frac {L_f}2}\epsilon_k\leq 1$ and $t_{k+1}^2\gamma_{k+1}\leq t_k^2\gamma_k$ for any $k\geq 1$.
We also have $\bar\epsilon_k=\sum_{i=1}^k\epsilon_i\leq \sqrt{\frac {L_f}2}D$ and $\sum_{i=1}^k\epsilon_i^2\leq \frac{L_fD^2}{4}$.

First, we discuss the upper bound of $\omega_k$. By definition, 
\[
    \|u_k\|=\|t_k(x_k-x^\star)-(t_k-1)(x_{k-1}-x^\star)\|\geq t_k\omega_k-(t_k-1)\omega_{k-1},
\]
where the last inequality holds because $t_k\geq 1$.
By \eqref{eqn:repeat}, it holds $\sqrt{b_k}=\sqrt{\frac {L_f}2}\|u_k\|\leq\sqrt{\tau}+\sqrt{s_k}$.
It follows
\[
    t_k\omega_k-(t_k-1)\omega_{k-1}\leq\sqrt{\frac 2{L_f}}(\sqrt{\tau}+\sqrt{s_k}),
\]
which implies
\[
    \omega_k\leq\omega_{k-1}+\sqrt{\frac 2{L_f}}(\sqrt{\tau}+\sqrt{s_k}).
\]
Substituting \eqref{eqn:sqrtsk} into the above inequality and combining \eqref{eqn:xik}, we get
\begin{align*}
    \omega_k & \leq \omega_{k-1}+\sqrt{\frac{2\tau}{L_f}}+\sqrt{\frac1{2{L_f}}}\left(\bar{\epsilon}_k+\left(\bar{\epsilon}_k^2+4 \bar{\xi}_{k-1}+4\xi_k+4\epsilon_k^2+4 \bar{\epsilon}_k \sqrt{\tau}\right)^{1 / 2}\right)\\
    &\leq \omega_{k-1}+\sqrt{\frac{2\tau}{L_f}}+\sqrt{\frac1{2{L_f}}}\bar{\epsilon}_k+\sqrt{\frac1{2{L_f}}}\left(\bar{\epsilon}_k^2+4 \bar{\xi}_{k-1} \right.\\
    &\left.\qquad\qquad+\frac{8L_f\LA t_{k-1}^2\gamma_{k-1}}{\sigmin^2}\omega_{k}+\frac{8(4L_f+\LA )t_{k-1}^2\gamma_{k-1}}{\sigmin^2}+4\epsilon_k^2+4 \bar{\epsilon}_k \sqrt{\tau}\right)^{1 / 2}.
\end{align*}
For the simplicity of notation, we denote $C_0 \defeq \sqrt{\frac{2\tau}{L_f}}+\sqrt{\frac1{2{L_f}}}\bar{\epsilon}_k$,
$C_1 \defeq \frac{4\LA t_{k-1}^2\gamma_{k-1}}{\sigmin^2}$, \\
$C_2 \defeq \frac{1}{2L_f}\left(\bar{\epsilon}_k^2+4 \bar{\xi}_{k-1} +\frac{8(4L_f+\LA )t_{k-1}^2\gamma_{k-1}}{\sigmin^2}+4\epsilon_k^2+4 \bar{\epsilon}_k \sqrt{\tau}\right)$,
and the above inequality can be rewritten into
\[
    \omega_k\leq\omega_{k-1}+C_0+\sqrt{C_1}\left(\frac{C_2}{C_1}+\omega_k\right)^{1 / 2},
\]
which implies
\[
    \left(\frac{C_2}{C_1}+\omega_k\right)^{1 / 2}\leq \frac12\left(\sqrt{C_1}+\sqrt{C_1+4\omega_{k-1}+4C_0+4C_2/C_1}\right).
\]
By simple derivations, we have
\begin{equation}\label{eqn:omegak}
    \begin{split}
    \omega_k & \leq \frac{C_1}{2}+\omega_{k-1}+C_0+\frac12\sqrt{C_1^2+4C_1\omega_{k-1}+4C_0C_1+4C_2}\\
    & \leq \omega_{k-1}+C_1+C_0+\sqrt{C_0C_1}+\sqrt{C_1\omega_{k-1}}+\sqrt{C_2}.
    \end{split}
\end{equation}
On the one hand, we observe that the upper bound of $\omega_k$ depends on $C_0$ and $C_2$, which further depend on $\bar\xi_{k-1}$.
On the other hand, from \eqref{eqn:xik}, we know $\xi_k$ can be bounded by $\omega_k$.
Observing the relationship, we can prove the upper bound of both $\omega_k$ and $\xi_k$ by induction.
It is easy to derive that 
$C_0 \leq \bar C_0\defeq\frac{3D}{2}$,
$C_1 \leq \bar C_1\defeq\frac{4\LA }{\sigmin^2}$ and
$C_2 \leq \frac{1}{2L_f}\left(3L_fD^2+4 \bar{\xi}_{k-1} +\frac{8(4L_f+\LA )}{\sigmin^2}\right)$.
Let 
$\bar C_2\defeq \frac{1}{2L_f}\left(6 L_fD^2 +\frac{8(4L_f+\LA )}{\sigmin^2}\right)$,
$D_0=\bar C_0+ \bar C_1+\sqrt{\bar C_2}+\sqrt{\bar C_0 \bar C_1}, D_1=\frac{4\LA }{\sigmin^2}$. Let $\bar\omega_1$ be the maximum zero point of \eqref{omega1eqn}
and 
\begin{equation}\label{def:varpi}
    \varpi=\max\left\{\frac{\left(\sqrt{D_1}+\sqrt{D_1+4D_0}\right)^2}4, \bar\omega_1\right\}.
\end{equation}
Note that $\bar C_0, \bar C_1, \bar C_2, D_0, D_1, \bar \omega_1$ and $\varpi$ only depend on the constants $\LA , \sigmin, L_f$ and $D$.
We aim to prove $\omega_k\leq k\varpi $ and $\xi_k\leq \frac{L_fD^2}{4k^2}$.

First, we prove the case when $k=1$.
Note that $\|u_1\| = \|x_1-x^\star\|=\omega_1$.
By \eqref{eqn:repeat}, we have
\[
    \omega_1\leq\sqrt{\frac 2{L_f}}(\sqrt{\tau}+\sqrt{s_1})
\]
Combining with \eqref{eqn:s1} yields
\begin{equation}\label{eqn:omega1}
    \omega_1\leq\sqrt{\frac 2{L_f}}\left(\sqrt{\tau}+\left(\epsilon_1^2+\xi_1+\epsilon_1\left(\sqrt{\tau}+\sqrt{\xi_1}\right)\right)^{1/2}\right).
\end{equation}
Putting \eqref{eqn:omega1} and \eqref{eqn:xi1} together and combining $\epsilon_1\leq\sqrt{\frac2{L_f}}$ yield a quartic inequality equation with respect to $\omega_1$:
\begin{equation}\label{omega1eqn}
    \begin{split}
        \omega_1\leq\sqrt{\frac 2{L_f}} \left[\sqrt{\tau}+\left(\frac2{L_f}+\frac{1}{\sigmin^2}(L_f\LA \omega_1+L_f(\LA D+1)+\LA )\right.\right.\\
        \left.\left. +\sqrt{\frac2{L_f}}\left(\sqrt{\tau}+\sqrt{\frac{1}{\sigmin^2}(L_f\LA \omega_1+L_f(\LA D+1)+\LA )}\right)\right)^{1/2}\right].
    \end{split}
\end{equation}
Recall the definition of $\bar \omega_1$ and $\varpi$. We get $\omega_1\leq\bar\omega_1\leq\varpi$.
Thus, it holds
\[
    \xi_{1}\leq \frac1{\underline{\sigma}^2_{\min}}(L_f\LA \varpi+L_f(\LA D+1)+\LA ) t_1^2\gamma_1\leq \frac{L_fD^2}{4},
\]
where the last inequality holds because $\gamma_{k}\leq \frac{\sigmin^2L_fD^2}{8(L_f\LA (\varpi+D)+4L_f +\LA )}\frac{1}{t_{k}^2(k+1)^3}$.

Suppose that there exists $k\geq2$ such  $\omega_i\leq i\varpi $ and $\xi_i\leq \frac{L_fD^2}{4i^2}$ for any $1\leq i\leq k-1$.
Then we have $\sum_{i=1}^{k-1}\xi_i\leq\frac{L_fD^2}{2}$ and $\bar{\xi}_{k-1}=\sum_{i=1}^{k-1}\xi_i+\sum_{i=1}^{k-1}\epsilon^2\leq \frac{3L_fD^2}4$.
Hence $C_2 \leq \bar C_2$ and \eqref{eqn:omegak} implies
\begin{align*}
\omega_k &\leq \omega_{k-1}+\bar C_1+\bar C_0+\sqrt{\bar C_0\bar C_1}+\sqrt{C_1\omega_{k-1}}+\sqrt{\bar C_2}\\
&\leq (k-1)\varpi+\bar C_1+\bar C_0+\sqrt{\bar C_0\bar C_1}+\sqrt{(k-1)C_1\varpi}+\sqrt{\bar C_2}\\
&\leq (k-1)\varpi+D_0+\sqrt{ D_1\varpi}\\
&\leq k\varpi,
\end{align*}
where the third inequality is due to $C_1\omega_{k-1}\leq \frac{4\LA }{\sigmin^2}t_{k-1}^2\gamma_{k-1}\cdot(k-1)\varpi\leq D_1\varpi$
and the last inequality holds because $\varpi$ is greater than the maximum zero point of the equation $D_0+\sqrt{D_1x}= x$.
Then we can obtain $\xi_k\leq \frac{L_fD^2}{4k^2}$ due to \eqref{eqn:xik} and $\gamma_{k}\leq \frac{\sigmin^2L_fD^2}{8(L_f\LA (\varpi+D)+4L_f +\LA )}\frac{1}{t_{k}^2(k+1)^3}$.
Consequently, by induction we conclude that $\xi_k\leq \frac{L_fD^2}{4k^2}$ for any $k\geq 1$. It follows $\bar{\xi}_k\leq L_fD^2$.
Combining the fact that $t_k \geq(k+1)/2$ and inequality \eqref{eqn:ak} yields
\[
    f(x_k)-f(x^\star)+\langle \lambda^\star, Ax_k-b\rangle \leq \frac{\left(\sqrt{\frac {L_f}2}D+\sqrt{\frac {L_f}2}D\right)^2+2L_fD^2}{(k+1)^2/4}\leq\frac{16L_fD^2}{(k+1)^2},
\]
which completes the proof. \qed
\end{proof}

\begin{corollary}
    Under the same assumptions and same choices of the parameters $\epsilon_k$ and $\gamma_k$ in Theorem \ref{thm:convex}, in order to find an approximate solution $x_k$ satisfying $v_k=f(x_k)-f(x^\star)+\langle \lambda^\star, Ax_k-b\rangle\leq \epsilon$,
    the total number of gradient evaluations for Algorithm \ref{APPA-convex} is bounded by
    \[
        \tilde {O}\left(\kappa_A\sqrt{\frac{L_fD^2}{\epsilon}}\right).
    \]
\end{corollary}
\begin{proof}
    In Theorem \ref{thm:convex}, we have proved the outer loop complexity is $O\left(\sqrt{\frac{L_fD^2}{\epsilon}}\right)$.
    Let $x_k^\star=\argmin_{Ax=b} \left\{f(x)+\frac {L_f}2\|x-y_k\|^2\right\}$.
    The KKT conditions are $\nabla f(x_k^\star)+L_f(x_k^\star-y_k)+A^\T\lambda_k^\star=0$ and $Ax_k^\star-b = 0$.
    By the definition of $\delta_k$ and $\zeta_k$, we have
    \begin{align*}
        \delta_k &= \nabla f(x_k)-\nabla f(x_k^\star)+L_f(x_k-x_k^\star)+A^\T(\lambda_k-\lambda_k^\star),\\
        \zeta_k & = A(x_k-x_k^\star).
    \end{align*}
    It follows $\|\delta_k\| \leq 2L_f\|x_k-x_k^\star\|+\LA \|\lambda_k-\lambda_k^\star\|$ and
    $\|\zeta_k\|  \leq \LA \|x_k-x_k^\star\|$.
    Let $\tilde\epsilon_k =\frac{1}{\LA +2L_f}\cdot\min\left\{\sqrt{\frac {L_f}2}\cdot\frac{\epsilon_k}{t_k}, \gamma_k\right\}$,
    then the subroutine only needs to output a pair $(x_k,\lambda_k)$ such that $\|x_k-x_k^\star\|\leq \tilde\epsilon_k$ and $\|\lambda_k-\lambda_k^\star\|\leq \tilde \epsilon_k$,
    the required subproblem error tolerance can be satisfied.
    The condition number of the objective function of the subproblem \eqref{proximal-point} is $O(1)$.
    Since $\tilde\epsilon_k$ is the power function of $\epsilon$ and other parameters, we obtain the number of gradient evaluations in each inner iteration is $\tilde{O}(\kappa_A)$ by Corollary \ref{coro:strongly-convex}.
    Consequently, the overall complexity is $\tilde {O}\left(\kappa_A\sqrt{\frac{L_fD^2}{\epsilon}}\right)$. \qed
\end{proof}
\begin{remark}
Note that Algorithm \ref{APPA-convex} outputs an approximate solution of the last subproblem, which satisfies $\|Ax_T-b\|\leq\gamma_T$. If we further require that $\gamma_T\leq\min\left\{\frac{\epsilon}{2\|\lambda^\star\|},\epsilon\right\}$, the overall complexity remains unchanged up to logarithmic factors since the inner loop converges linearly. Therefore, in order to obtain an approximate solution $x_T$ satisfying $f(x_T)-f(x^\star)+\langle \lambda^\star, Ax_T-b\rangle\leq \frac\epsilon2$ and $\|Ax_T-b\|\leq\gamma_T$, the complexity upper bound is still $\tilde {O}\left(\kappa_A\sqrt{\frac{L_fD^2}{\epsilon}}\right)$. In this case, we can conclude that  $f(x_T)-f(x^\star)\leq \epsilon$ because $\gamma_T\leq \frac{\epsilon}{2\|\lambda^\star\|}$. Therefore, the above result actually implies an complexity upper bound of $\tilde {O}\left(\kappa_A\sqrt{\frac{L_fD^2}{\epsilon}}\right)$ in order to find an approximate solution $x_T$ satisfying $f(x_T)-f(x^\star)\leq\epsilon$ and $\|Ax_T-b\|\leq\epsilon$.
\end{remark}

\section{Conclusions}
In this work, we analyze the lower and upper bounds of composite optimization problems in strongly convex, convex, and non-convex scenarios. Different from most of previous studies, we specifically consider problem classes with a given condition number $\kappa_A$. Our results demonstrate that the complexities presented are optimal up to logarithmic factors. This study marks the first instance of optimal algorithms for convex and non-convex cases, as well as the first set of lower bounds for all three cases. In the future work, it remains interesting to investigate how algorithms can be designed to further tighten the logarithmic factors in the complexity upper bounds.

\vspace{0.3cm}
{\small 
\noindent\textbf{Funding} This work was supported by  the National Natural Science Foundation of China under grant number 11831002.\\

\noindent\textbf{Data Availability} \\

\noindent\textbf{Declarations}\\

\noindent\textbf{Conflict of interest} The authors have no relevant financial interest to disclose.
}

\bibliographystyle{spmpsci}
\bibliography{references}

\end{document}

%% file: command.tex
\newcommand{\sigmin}{\underline{\sigma}_{\min}}
\newcommand{\sigmamin}{\mu_A} 
\newcommand{\muA}{\mu_A}
\newcommand{\LA}{L_A}

\newcommand{\muf}{\mu_f}
\newcommand{\muphi}{\mu_{\Phi}}

\newcommand{\defeq}{\mathrel{\mathop:}=}
\newcommand{\supp}{\operatorname{supp}}
\newcommand{\suboptsc}{\operatorname{SubOpt}_\mathsf{SC}}
\newcommand{\suboptc}{\operatorname{SubOpt}_\mathsf{C}}
\newcommand{\suboptnc}{\operatorname{SubOpt}_\mathsf{NC}}

\newcommand{\st}{\,\textrm{s.t.}\,}  
\newcommand{\T}{\mathrm{T}}

\newcommand{\RR}{\mathbb{R}}

\newtheorem{assumption}[theorem]{Assumption}

\newcounter{cnt}
\foreach \num in {1,2,...,26}{%
  \stepcounter{cnt}%
  \expandafter\xdef \csname c\Alph{cnt}\endcsname {\noexpand\mathcal{\Alph{cnt}}}%
  \expandafter\xdef \csname b\Alph{cnt}\endcsname {\noexpand\mathbb{\Alph{cnt}}}%
}

\newcommand{\spa}{\mathrm{span}} 
 
\newcommand{\prox}{\mathbf{prox}} 
\newcommand{\proxhAx}{\prox_{\frac{1}{L_f}h(A\cdot-b)}}
\newcommand{\iprod}[2]{\left\langle{#1},{#2}\right\rangle}
\newcommand{\nrm}[1]{\left\|#1\right\|}

\newcommand{\dom}{\mathrm{dom}}

\newcommand{\abs}[1]{\left|#1\right|}

\DeclareMathOperator*{\argmin}{arg\,min}
\DeclareMathOperator*{\argmax}{arg\,max}

\newcommand{\paren}[1]{\left(#1\right)}
\newcommand{\brac}[1]{\left[#1\right]}
\newcommand{\set}[1]{\left\{#1\right\}}

\newcommand{\hx}{\hat{x}}
\newcommand{\xs}{x^\star}
\newcommand{\lams}{\lambda^\star}

\newcommand{\nrmA}{\nrm{A}}